\documentclass[11pt]{amsart}
\usepackage{fullpage,amsmath,amsthm,amsfonts,amssymb,graphicx,amscd,float,picins,hyperref,
setspace}
\usepackage[shortlabels]{enumitem}
\usepackage{color, xcolor}
\usepackage{tikz-cd}
\usepackage{xypic}
\usetikzlibrary{cd}
\usepackage{wrapfig}
\usepackage{eucal}

\definecolor{ivory}{rgb}{1.0, 1.0, 0.94}
\definecolor{carnelian}{rgb}{0.7, 0.11, 0.11}
	\definecolor{cinnamon}{rgb}{0.82, 0.41, 0.12}

\title{IW4}

\newtheorem{thm}{Theorem}[section]

\newtheorem{lem}[thm]{Lemma}

\newtheorem{prop}[thm]{Proposition}

\theoremstyle{definition}

\newtheorem{rmk}{Remark}

\newcommand{\out}{\textup{Out}(F_r)}
\newcommand{\outt}{\textup{Out}(F_3)}

\newcommand{\os}{\textup{CV}_r}

\newcommand{\RR}{\mathbb R}

\newcommand{\ol}{\overline}
\newcommand{\mG}{\mathcal{G}}
\newcommand{\mI}{\mathcal{I}}
\newcommand{\mA}{\mathcal{A}}
\newcommand{\mE}{\mathcal{E}}
\newcommand{\mV}{\mathcal{V}}
\newcommand{\mD}{\mathcal{D}}

\newcommand{\mF}{\mathcal{F}}

\newcommand{\mL}{\mathcal{L}}

\newcommand{\G}{\Gamma}
\newcommand{\ZZ}{\mathbb{Z}}

\newcommand{\from}{\colon}

\newcommand{\al}{\alpha}
\newcommand{\vphi}{\varphi}

\newcommand{\lam}{\lambda}

\newcommand{\im}{\text{Im}}

\DeclareMathOperator{\LW}{LW}
\DeclareMathOperator{\SW}{SW}
\DeclareMathOperator{\IW}{IW}

\theoremstyle{definition}
\newtheorem{rk}[thm]{Remark}
\newtheorem{ex}[thm]{Example}

\theoremstyle{plain}

\newtheoremstyle{TheoremNum}
{8.0pt plus 2.0pt minus 4.0pt}{8.0pt plus 2.0pt minus 4.0pt} %%% space between body and thm
{\itshape} %%% Thm body font
{-0.15cm} %%% Indent amount (empty = no indent)
{\bfseries} %%% Thm head font
{.} %%% Punctuation after thm head
{ }  %%% Space after thm head
{\thmname{#1}\thmnote{ \bfseries #3}}%%% Thm head spec
\theoremstyle{TheoremNum}
\newtheorem{duplicate}{}

\usepackage{tikz}
\usetikzlibrary{decorations.markings}

\tikzset{middlearrow/.style n args={4}{
		decoration={
			markings,
			mark=at position #1 with {\arrow{#2},\node[transform shape,#4] {#3};}},postaction={decorate}},
	middlearrow/.default={.5}{>}{}{below}	}
\usetikzlibrary{calc}

\begin{document}

\pagecolor{ivory}

\title{$\out$ train track automata I:\\ Proper full fold decompositions}
\author{Catherine Eva Pfaff}

\address{\tt Department of Mathematics \& Statistics, Queen's University
  \newline \indent  {\url{https://mast.queensu.ca/~cpfaff/}}, } \email{\tt catherine.pfaff@gmail.com}

\date{}
\maketitle

\begin{abstract}
We describe train track automata for large classes of fully irreducible elements of $\out$, and their associated geodesics in Culler-Vogtmann Outer Space.
\end{abstract}

\section{Introduction}

Let $F_r$ denote the free group of rank $r \geq 3$, and consider its outer automorphism group $\out$.  $\out$ acts isometrically \cite{FrancavigliaMartino}, with finite point stabilizers and North-South dynamics \cite{ll03}, on the parameter space of rank-$r$ weighted graphs, the Culler-Vogtmann Outer space $CV_r$ \cite{cv86}. The dynamically minimal and generic elements of $\out$ are the fully irreducible elements and associated to each fully irreducible element is a collection of invariant axes. These axes arise from the Stallings fold decompositions \cite{s83} of special topological representatives called ``train track maps'' \cite{bh92}.

\cite{stablestrata} proposes a stratification of fully irreducible axes by their \cite{hm11} ideal Whitehead graphs, much as there is a stratification of pseudo-Anosov axes by the index lists of their invariant foliations/ laminations. This paper is the first of a series of papers describing and analyzing the train track automata for these ideal Whitehead graph ``strata.'' Ideal Whitehead graphs are $\out$ conjugacy class invariants and can be seen to describe the behavior of lamination leaves at the singularities in the attracting trees in $\partial\os$.

We start with the train track automata for ``lone axis'' fully irreducible outer automorphisms, as first introduced in \cite{loneaxes} and characterized by having only a single invariant axis. These automata generalize those constructed in \cite{IWGI}, \cite{automaton}, and \cite{wiggd1} to encode all lone axis fully irreducible train track representatives. To avoid overloading the introduction with definitions, we include here only an abbreviation of Theorem \ref{t:LAautomata_loops_represent1} and Theorem \ref{t:LAautomata_loops_represent2}:

\begin{duplicate}[Theorem~\ref{t:LAautomata_loops_represent1}-\ref{t:LAautomata_loops_represent2}]
{\begin{itemize} The following 2 statements hold:
  \item[(a.)] Any loop in the lone axis automaton $\mA(\mG)$ represents a train track map and, under suitable conditions, this train track map represents a fully irreducible $\vphi\in\out$ such that the ideal Whitehead graph of $\vphi$ is isomorphic to $\mG$.
  \item[(b.)]  Every train track representative of a lone axis fully irreducible $\vphi\in\out$ determines a directed loop in the lone axis train track automaton $\mA(IW(\vphi))$.
\end{itemize}}
\end{duplicate}

As one moves away from lone axis fully irreducible outer automorphisms, the situation becomes substantially more complicated. A display of this is included in Example \ref{e:PathologicalExample}. While all of the ``pathologies'' of Example \ref{e:PathologicalExample} can still occur in such a setting, we tackle them in constructing the train track automata for what we call ``proper full fold (pff) decompositions'' of ``fully singular'' train track representatives.

As above, we include here only an abbreviation of Theorems \ref{t:Pffautomata_loops_represent1}, \ref{t:Pffautomata_loops_represent2}, and
\ref{t:PathsAreGeodesic}:

\begin{duplicate}[Theorem~\ref{t:Pffautomata_loops_represent1} -\ref{t:Pffautomata_loops_represent2}]
{\begin{itemize} The following 3 statements hold:
  \item[(a.)] Any loop in the fully singular pff train track automaton $\mA(\mG)$ represents a train track map and, under suitable conditions, this train track map represents a fully irreducible $\vphi\in\out$ such that the ideal Whitehead graph of $\vphi$ is isomorphic to $\mG$.
  \item[(b.)] Suppose $g$ is a fully singular train track representative of a fully irreducible $\vphi\in\out$. Then any pff decomposition of $g$ determines a directed loop in a fully singular pff train track automaton, more precisely $\mA(IW(\vphi))$.
  \item[(c.)] Any bi-infinite path in a fully singular pff train track automaton $\mA(\mG)$ determines a geodesic in $\os$.
\end{itemize}}
\end{duplicate}

The axes in $\os$ of proper full fold decompositions of fully singular train track representatives avoid the merging of vertices. Algom-Kfir, Kapovich, and Pfaff provide in \cite[\S 6]{stablestrata} an example of a merging of vertices during a fully irreducible axis. In Example \ref{ex:StallingsFolds} we provide a fully irreducible outer automorphism where only the proper choice of axis avoids this phenomena and \S \ref{ss:MaximallySingular} provides some analysis of its occurrence. We stop short of fully tackling such circumstances in this paper by focusing on the automata for proper full fold decompositions of fully singular train track representatives.

We conjecture that each ageometric fully irreducible outer automorphism has a fully singular train track representative and these representatives generically have proper full fold decompositions.

\subsection*{Acknowledgements}
The author is grateful to Ilya Kapovich and Lee Mosher for inspiring conversations and continued interest in her work, as well as to Chi Cheuk Tsang, Paige Hillen, and the entire WiGGD group for being such joys to work with that this paper was largely written to help progress both our joint work and your individual projects. This paper is for you. Thanks are owed to Naomi Andrew for assisting with certain computations and inspiring ideas, and to Chi Cheuk Tsang for inspiring ideas and reading an earlier draft. The author is grateful to the Institute for Advanced study for their hospitality, and Bob Moses for funding her 2024-2025 membership. This work has received funding from an NSERC Discovery Grant.

%%%%%%%%%%%%%%%%%%%%%%%%%%%%%%%%%%%%%%%%%%%%%%%%%%%%%%%%%%%%%%%%%%%%
%%%%%%%%%%%%%%%%%%%%%%%%%%%%%%%%%%%%%%%%%%%%%%%%%%%%%%%%%%%%%%%%%%%%

%%%%%%%%%%%%%%%%%%%%%%%%%%%%%%%%%%%%%%%%%%%%%%%%%%%%%%%%%%%%%%%%%%%%%%%
%%%%%%%%%%%%%%%%%%%%%%%%%%%%%%%%%%%%%%%%%%%%%%%%%%%%%%%%%%%%%%%%%%%%%%%%%

\section{Background}

%%%%%%%%%%%%%%%%%%%%%%%%%%%%%%%%%%%%%%%%%%%%%%%%%%%%%%%%%%%%%%%%%%%%%%%%

Assume throughout this section that $\G$ is a finite oriented graph where each vertex has valence $\geq 3$ and $F_r$ is a free group of rank $r\geq 3$. This section will contain explanations of background given in previous works, included here for convenience. Expository overlap may occur.

%%%%%%%%%%%%%%%%%%%%%%%%%%%%%%%%%%%%%%%%%%%%%%%%%%%%%%%%%%%%%%%%%%%%%%%%

\subsection{Edge Maps on Graphs}{\label{s:graphmaps}}

%%%%%%%%%%%%%%%%%%%%%%%%%%%%%%%%%%%%%%%%%%%%%%%%%%%%%%%%%%%%%

Suppose $\G$ has positively oriented edges $\{e_1, e_2 \dots, e_n\}$ and vertices $\{v_1, v_2, \dots, v_m\}$. We use the notation $E\G:=\{e_1, e_2 \dots, e_n\}$ for the positively oriented edge set, and $V\G :=\{v_1, v_2, \dots, v_m\}$ for the vertex set, and $E^{\pm}\G = \{e_1, \overline{e_1}, \dots, e_n, \overline{e_n}\}$ for the full directed edge set, with an overline indicating a reversal of orientation.

%%%%%%%%%%%%%%%%%%%%%%%%%%%%%%%%%%%%%%%%%%%%%%%%%%%%%%%%%%%%%

\subsubsection{Directions, turns, \& edge paths}{\label{ss:Turns}}
		
Given $v\in V\G$, a \textit{direction at $v$} is an element of $E^{\pm}\G$ with initial vertex $v$. For each $v\in V\G$, let $\mD(v)$ denote the set of directions at $v$ and $\mD\G :=\cup_{V\G}\mD(v)$. A \textit{turn at $v$} is an unordered pair $\{d_1, d_2\}$ of directions at $v$ and is \emph{degenerate} if $d_1=d_2$.

\vspace*{-1.5mm}

\parpic[r]{\includegraphics[width=1.1in]{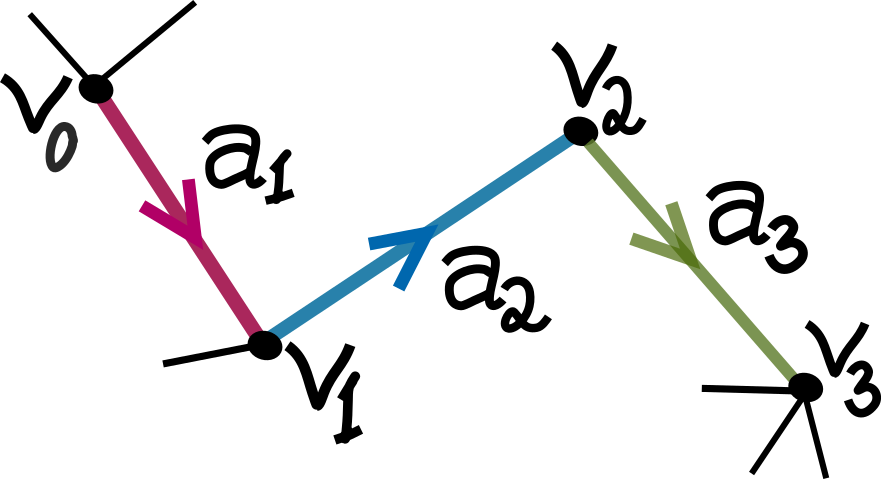}} 
An \textit{edge path} (or just \textit{path}) $\rho$ in $\G$ is a finite sequence $(a_1, a_2, \dots, a_{\ell})\in (E^{\pm}\G)^{\ell}$ such that there
exists a sequence $(v_1, v_2, \dots, v_{\ell-1})\in (V\G)^{\ell-1}$ satisfying that the turn $\{\overline{a_{j}}, a_{j+1}\}$ is a turn at $v_j$ for each $j\in\{1,2,\dots,\ell-1\}$. For such a path $(a_1, a_2, \dots, a_m)$ we write $\gamma = a_1a_2\dots a_m$ and say $\gamma$ \textit{contains} the oriented edges $a_1, a_2, \dots, a_m$ and \textit{takes} the turns $\{\overline{a_1}, a_2\}$, $\{\overline{a_2}, a_3\}$, $\dots$, $\{\overline{a_{n-1}}, a_n\}$. Note that $\ol{\gamma} = \ol{a_1}~ \ol{a_2}\dots \ol{a_m}$ take the same set of turns. We call $\gamma$ \textit{tight} if it takes no degenerate turns, colloquially described as there being no ``backtracking.'' Fixing endpoints, there is a unique tight path in the homotopy class of each path $\gamma$. We say this path is obtained from $\gamma$ by \emph{tightening} $\gamma$.

%%%%%%%%%%%%%%%%%%%%%%%%%%%%%%%%%%%%%%%%%%%%%%%%%%%%%%%%%%%%%

\subsubsection{Graph maps}{\label{ss:GraphMaps}}

An \textit{edge (or graph) map} $g:\G\to\G'$ is
\begin{itemize}
\item a map $\mV:V\G\to V\G'$, where we write $g(v)$ for $\mV(v)$, together with
\item for each $e\in E^{\pm}\G$, an assignment of a path $g(e)$ in $\G'$ such that
\begin{enumerate}
\item if the initial vertex of $e$ is $v$, then the initial vertex of $g(e)$ is $g(v)$, and
\item if $g(e)$ is the edge path $g(e) = a_1a_2\dots a_m$, then $g(\overline{e})$ is the edge path $g(\overline{e}) = \overline{a_m}\dots\overline{a_2}~\overline{a_1}$.
\end{enumerate}
\end{itemize}
\noindent Viewing $\Gamma$ and $\Gamma'$ as topological spaces, $g$ is a continuous map sending vertices to vertices. 

$g_{\#}(e)$ is obtained from $g(e)$ by tightening, as in \S \ref{ss:Turns}.

%%%%%%%%%%%%%%%%%%%%%%%%%%%%%%%%%%%%%%%%%%%%%%%%%%%%%%%%%%%%%

\subsubsection{Taken turns \& tight maps}{\label{ss:TakenTurns}}

Given a graph map $g\colon\G\to\G$, we define $$\tau(g):=\{\{d_1,d_2\}\mid\{d_1,d_2\} \text{ is taken by } g(e) \text{ for some } e\in E\G \}$$ and say each turn $T\in\tau(g)$ is \emph{$g$-taken}. Let $$\tau_{\infty}(g)=\bigcup_{k=1}^{\infty}\tau(g^k)$$
\noindent and say each turn $T\in\tau_{\infty}(g)$ is \emph{$g_{\infty}$-taken}.

We then call $g$ \emph{tight} if the image of each edge is a tight path. In particular, no degenerate turns are $g$-taken. If $\gamma = a_1a_2\dots a_n$ is a path in $\G$ for some $a_1, a_2, \dots, a_n \in E^{\pm}\G$, then $g(\gamma)$ will mean the concatenation of edge paths $g(\gamma) = g(a_1)g(a_2)\dots g(a_n)$. Note that $g(\gamma)$ is tight if and only if $\gamma$ is tight and $g$ is locally injective on $\gamma$.

%%%%%%%%%%%%%%%%%%%%%%%%%%%%%%%%%%%%%%%%%%%%%%%%%%%%%%%%%%%%%

\subsubsection{Directions maps \& gate structures}{\label{ss:GateStructures}}

To $g$ we associate a \emph{direction map} $Dg:\mD\G\to\mD\G'$ such that if $g(e) = a_1 a_2\dots a_m$, for some $m\geq 1$ and $a_1,a_2,\dots,a_m \in E^{\pm}\G$, then $Dg(e) = a_1$. 
We call a direction $e$ \emph{periodic} if $Dg^k(e) = e$ for some $k>0$, and \emph{fixed} if $k=1$. When $g\from \G \to \G$ is a self-map, the turn $\{d_1,d_2\}$ is an \emph{illegal turn} for $g$ if $\{Dg^k(d_1),Dg^k(d_2)\}$ is degenerate for some $k\in\ZZ_{\geq 1}$. Defining an equivalence relation on $\mD\G$ by $d_1\sim d_2$ when $\{d_1,d_2\}$ is an illegal turn, the equivalence classes are \emph{gates} and the partitioning of the directions at each vertex into gates is the \emph{induced gate structure}.
Note that each gate at a periodic vertex contains a unique periodic direction.

%%%%%%%%%%%%%%%%%%%%%%%%%%%%%%%%%%%%%%%%%%%%%%%%%%%%%%%%%%%%%

\subsubsection{Markings \& representatives}{\label{ss:RotationlessPowers}}

Viewing $g$ as a continuous map of graphs, we say $g$ \emph{represents} $\vphi$ when $\pi_1(\G)$ has been identified with $F_r$ (i.e. $\G$ is \emph{marked}) and $\vphi$ is the induced map of fundamental groups. When a marking is not explicitly given, we mean ``there exists a marking such that.''

\vskip10pt

%%%%%%%%%%%%%%%%%%%%%%%%%%%%%%%%%%%%%%%%%%%%%%%%%%%%%%%%%%%%%
%%%%%%%%%%%%%%%%%%%%%%%%%%%%%%%%%%%%%%%%%%%%%%%%%%%%%%%%%%%%%%%%%%%%%%%%

\subsection{Train track (tt) maps and fully irreducible outer automorphisms}{\label{s:fi}}

%%%%%%%%%%%%%%%%%%%%%%%%%%%%%%%%%%%%%%%%%%%%%%%%%%%%%%%%%%%%%%%%%%%%%%%%

\subsubsection{Train track (tt) maps}{\label{ss:TtMaps}}

	Suppose $g:\G\to\G$ is an edge map. We call $g$ a \textit{train track (tt) map} if $g^k$ is tight for each $k\in\ZZ_{>0}$. We assume throughout that tt maps are surjective, as they will automatically be for homotopy equivalences, due to our vertex valence restriction at the start of the section. We call the train track map $g$ \textit{expanding} if for each edge $e \in E\G$ we have $|g^n(e)|\to\infty$ as $n\to\infty$, where for a path $\gamma$ we use $|\gamma|$ to denote the number of edges $\gamma$ traverses (with multiplicity). Note that, apart from our not requiring a ``marking,'' these definitions coincide with those in  \cite{bh92} when $g$ is in fact a homotopy equivalence of graphs (viewed topologically).

%%%%%%%%%%%%%%%%%%%%%%%%%%%%%%%%%%%%%%%%%%%%%%%%%%%%%%%%%%%%%

\subsubsection{Transition matrices}{\label{ss:TransitionMatrices}}

The \textit{transition matrix} $M(g)$ of a tt map $g\from\Gamma \to \Gamma$ is the square $|E\G| \times |E\G|$ matrix $[a_{ij}]$ such that $a_{ij}$, for each $i$ and $j$, is the number of times $g(e_i)$ contains either $e_j$ or $\overline{e_j}$. Note that each transition matrix is a nonnegative integer matrix.

A nonnegative integer matrix $A=[a_{ij}]$ is \textit{irreducible} if for each $(i,j)$, there is a $k\in\ZZ_{>0}$ so that the ${ij}^{th}$ entry of $A^k$ is positive, and so in particular is at least $1$.
%If $A^k$ is strictly positive for some $k\in\ZZ_{>0}$ then $A$ is \textit{primitive}. 
Furthermore, $A$ is \textit{Perron--Frobenius (PF)} if there exists an $N$ such that, for each $k \geq N$, we have that $A^k$ is strictly positive.

%%%%%%%%%%%%%%%%%%%%%%%%%%%%%%%%%%%%%%%%%%%%%%%%%%%%%%%%%%%%%

\subsubsection{(Fully) irreducible outer automorphisms}{\label{ss:FI}}

A tt map is \textit{irreducible} if its transition matrix is irreducible. Not every element of $\out$ is represented by a tt map, and even fewer by irreducible tt maps. An outer automorphism $\vphi\in\out$ is \emph{fully irreducible} if no positive power preserves the conjugacy class of a proper free factor of $F_r$. Bestvina and Handel \cite{bh92} proved that each fully irreducible outer automorphism has expanding irreducible tt representatives, with Perron--Frobenius transition matrix.

If $g$ is a tt representative of a fully irreducible $\vphi\in\out$, then $\tau_{\infty}(g)$ is the set of turns taken by the \cite{bfh97} ``stable lamination'' $\Lambda_{\vphi}$.

\vskip10pt

%%%%%%%%%%%%%%%%%%%%%%%%%%%%%%%%%%%%%%%%%%%%%%%%%%%%%%%%%%%%%%%%%%%%%%%%

\subsection{(Periodic) Nielsen paths}{\label{s:fi}}

%%%%%%%%%%%%%%%%%%%%%%%%%%%%%%%%%%%%%%%%%%%%%%%%%%%%%%%%%%%%%%%%%%%%%%%%

Let $g \colon \Gamma \to \Gamma$ be an expanding irreducible tt map. Bestvina and Handel \cite{bh92} define a nontrivial immersed path $\rho$ in $\Gamma$ to be a \emph{periodic Nielsen path (PNP)} if, for some power $R \geq 1$, we have $g^R(\rho) \cong \rho$ rel endpoints (and just a \emph{Nielsen path (NP)} if $R=1$). An NP $\rho$ is \emph{indivisible} (hence is an ``iNP'') if it cannot be written as $\rho = \gamma_1\gamma_2$, where $\gamma_1$ and $\gamma_2$ are themselves NPs. Bestvina and Handel describe in \cite[Lemma 3.4]{bh92} the structure of iNPs:

\begin{lem}[\cite{bh92}]{\label{l:iNP}}
Let $g \colon \Gamma\to\Gamma$ be an expanding irreducible train track map and $\rho$ an iNP for $g$. Then $\rho=\bar \rho_1\rho_2$, where $\rho_1$ and $\rho_2$ are nontrivial legal paths originating at a common vertex $v$ and such that the turn at $v$ between $\rho_1$ and $\rho_2$ is a nondegenerate illegal turn for $g$.
\end{lem}

%%%%%%%%%%%%%%%%%%%%%%%%%%%%%%%%%%%%%%%%%%%%%%%%%%%%%%%%%%%%%

\subsubsection{Rotationless powers}{\label{ss:RotationlessPowers}}

By \cite[Corollary 4.43]{fh11}, for each $r \geq 2$, there exists a \emph{rotationless} power $R(r) \in \ZZ_{>0}$ such that for each expanding irreducible tt representative $g$ of a fully irreducible $\vphi \in \out$, among other properties, each periodic vertex, direction, and PNP is fixed by $g^{R(r)}$.

\vskip10pt

%%%%%%%%%%%%%%%%%%%%%%%%%%%%%%%%%%%%%%%%%%%%%%%%%%%%%%%%%%%%%
%%%%%%%%%%%%%%%%%%%%%%%%%%%%%%%%%%%%%%%%%%%%%%%%%%%%%%%%%%%%%%%%%%%%

\subsection{PNP detection}{\label{ss:PNPelimination}}

%%%%%%%%%%%%%%%%%%%%%%%%%%%%%%%%%%%%%%%%%%%%%%%%%%%%%%%%%%%%%%%%%%%%

\subsubsection{(Dangerous) long turns}{\label{ss:LongTurns}}

The notion of a dangerous long turn is first introduced in \cite{longturns}. Suppose $g\colon\G\to\G$ on $\Gamma$ is a tt map. By a \emph{long turn} at a vertex $v\in V\G$ we mean a pair of legal paths $\{\alpha, \beta\}$ emanating from $v$. We call $\{\alpha, \beta\}$ \emph{legal} or \emph{illegal} when $\{D(\alpha), D(\beta)\}$ is, respectively, legal or illegal. A long turn $\{\alpha, \beta\}$ is \emph{dangerous} if we have that
\begin{itemize}
  \item $g(\alpha)$ is not an initial subpath of $g(\beta)$, and
  \item $g(\beta)$ is not an initial subpath of $g(\alpha)$, and
  \item $g_{\#}(\ol{\alpha} \beta)$ is an illegal path (i.e. cancellation of $g_{\#}(\alpha)$ and $g_{\#}(\beta)$ ends with an illegal turn).
\end{itemize} 

\subsubsection{Identifying PNPs}{\label{ss:IdentifyingPNPs}}

The following is \cite[Lemma 4.2]{stablestrata} and gives a means for identifying PNPs via dangerous long turns.

\begin{lem}[\cite{stablestrata}]{\label{l:longinps}}
Let $g \colon \Gamma\to\Gamma$ be an expanding irreducible tt map and $\rho$ an iNP for $g$. Then $\rho=\bar \rho_1\rho_2$, where $\{\rho_1,\rho_2\}$ is a dangerous long turn for each positive power $g^k$ of $g$. More generally, if $g \colon \Gamma\to\Gamma$ has a PNP, then $\Gamma$ contains dangerous long turns for each positive power $g^k$ of $g$. Thus, an expanding irreducible train track map with no dangerous long turns has no PNPs.
\end{lem}

We now prove a lemma that arises in various forms in other papers but is needed in this fuller generality here.

\begin{lem}{\label{l:pnpelimination}}
Suppose $g \colon \Gamma\to\Gamma$ is an expanding irreducible tt map and $g'\from\G\to\G'$ and $g''\from\G'\to\G$ are surjective tight graph maps such that $g = g''\circ g'$. Let $f = g'\circ g''\colon\G'\to\G'$ and suppose $\rho=\bar \rho_1\rho_2$ is an iNP for $g$. Then each of the following holds.
\begin{enumerate}[(a)]
  \item $f$ is an expanding irreducible tt map with PF transition matrix $M(f)$.
  \item If $\rho_1'\subset\rho_1$ and $\rho_2'\subset\rho_2$ are subpaths such that $g'_{\#}(\rho)=\ol{\rho_1'} \rho_2'$, then $\rho'$ is an iNP for $f$ and $\{\rho_1', \rho_2'\}$ is a dangerous turn for $f$.
\end{enumerate} 
\end{lem}

\begin{proof} 
(a) Suppose for the sake of contradiction that $f^k$ is not injective on the interior of some $e\in E\G'$. Since $g'$ is surjective $g'(e')$ contains $e$ for some $e'\in E\G^{\pm}$. Then there would be cancellation on the interior of $g^{k+1}(e')=g''(f^k(g'(e')))\supseteq g''(f^k(e))$, contradicting that $g$ is a tt map.

Since $g$ is expanding, there exists an edge $e\in E\G$ such that $|g^n(e)|\to\infty$ as $n\to\infty$. Since $g''$ is surjective, there exists an edge $e'\in E\G'$ so that $g''(e')$ contains $e'$ and thus $|f^n(e)|\to\infty$ as $n\to\infty$.

To show $M(f)$ is PF, it suffices to show that there exists a power $n$ of $f$ so that $f^n(e)$, for each $e\in E\G'$, contains either $e'$ or $\ol{e'}$ for each $e'\in E\G'$. Since $M(g)$ is PF, %for each pair of edges $e,e'\in E\G$, 
there exists a power $k$ so that each edge of $E\G$ contains an orientation of each edge of $E\G$ in its $g^k$-image. Then, for each $e\in E\G'$, we know $g^k(g''(e))$ contains an orientation of each edge of $E\G$ in its image. And, since $g'$ is surjective, $g'(g^k(g''(e)))=f^{k+1}(e)$ contains an orientation of each edge of $E\G'$ in its image.

(b) Let $\rho':=g'_{\#}(\rho)$. By the definitions and the fact that $\rho$ is an iNP for $g$,
$$f(\rho')\simeq f(g'(\rho))\simeq (f\circ g')(\rho)\simeq (g'\circ g''\circ g')(\rho)\simeq (g'\circ g)(\rho))\simeq g'(g(\rho))\simeq g'(\rho)\simeq \rho'.$$
So $\rho'$ is by definition an NP for $f$. By symmetric argumentation, a decomposition of $\rho'$ into Nielsen paths would yield a decomposition of $\rho$ into Nielsen paths, contradiction that $\rho$ is an iNP. Thus, $\rho'$ is in fact an iNP. Finally, $\{\rho_1', \rho_2'\}$ is a dangerous turn for $f$.
\end{proof}

\vskip10pt

%%%%%%%%%%%%%%%%%%%%%%%%%%%%%%%%%%%%%%%%%%%%%%%%%%%%%%%%%%%%%%%%%%%%
%%%%%%%%%%%%%%%%%%%%%%%%%%%%%%%%%%%%%%%%%%%%%%%%%%%%%%%%%%%%%%%%%%%%

\subsection{Whitehead graphs \& lamination train track (ltt) structures}{\label{ss:WGs}}

%%%%%%%%%%%%%%%%%%%%%%%%%%%%%%%%%%%%%%%%%%%%%%%%%%%%%%%%%%%%%

Local Whitehead graphs, stable Whitehead graphs, and ideal Whitehead graphs were first introduced by Handel and Mosher in \cite{hm11}. We stray from their definitions by assuming throughout that $g\colon\G \to \G$ is a tt map with no PNPs. However, the presence of PNPs only impacts the ideal Whitehead graph definition.

%%%%%%%%%%%%%%%%%%%%%%%%%%%%%%%%%%%%%%%%%%%%%%%%%%%%%%%%%%%%%

\subsubsection{Local \& stable Whitehead graph}{\label{ss:WGs}}

For each $v \in V\G$ the \textit{local Whitehead graph} $\LW(g,v)$ has vertices for the directions of $\mD(v)$ and edges connecting the directions of each turn in $\tau_{\infty}(g)$. Given a $g$-periodic $v\in V\G$, the \emph{stable Whitehead graph} $\SW(g,v)$ is the restriction of $\LW(g,v)$ to the periodic direction vertices and edges betwixt them. In terms of gates, $\SW(g,v)$ has a vertex for each gate at $v$.
	
%%%%%%%%%%%%%%%%%%%%%%%%%%%%%%%%%%%%%%%%%%%%%%%%%%%%%%%%%%%%%

\subsubsection{Ideal Whitehead graph}{\label{ss:IWGs}}

In the absence of PNPs, if $g$ represents a fully irreducible outer automorphism $\vphi$, then the \textit{ideal Whitehead graph} $\IW(\vphi)$ for $\vphi$ is defined as
$$\IW(\vphi)=\bigsqcup_{v\in V\Gamma} \SW(g,v),$$ 
but with components containing only two vertices left out. 

The ideal Whitehead graph is an invariant of the conjugacy class of the outer automorphism represented by $g$ and $\IW(\vphi^k)=\IW(\vphi)$ for each $k\in\ZZ_{>0}$ \cite{hm11, Thesis}.

%%%%%%%%%%%%%%%%%%%%%%%%%%%%%%%%%%%%%%%%%%%%%%%%%%%%%%%%%%%%%

\subsubsection{Lamination train track (ltt) structure $\mG(g)$}{\label{ss:LTTs}}

The \textit{lamination train track (ltt) structure} $\mG(g)$ is obtained from its \emph{underlying graph} $\G$ by replacing each vertex $v \in V\G$ with the local Whitehead graph $\LW(g,v)$ as follows. Replace $v$ with a vertex for each direction at $v$ labeled with that direction and then identify each of these new vertices with the vertex of $\LW(g,v)$ of the same label. Vertices and edges of each $\SW(g,v)$ are colored purple and the remaining vertices and (open) edges of each $\LW(g,v)$ are colored red. Alternatively, one could start with $\bigsqcup_{v\in V\G} \LW(g,v)$, color the $\LW(g,v)$ as just described, and then include a directed edge $[e,\ol{e}]$ for each $e\in E\G$. See \S \ref{s:ExPurity} for examples.

\vskip10pt

%%%%%%%%%%%%%%%%%%%%%%%%%%%%%%%%%%%%%%%%%%%%%%%%%%%%%%%%%%%%%%%%%%%%
%%%%%%%%%%%%%%%%%%%%%%%%%%%%%%%%%%%%%%%%%%%%%%%%%%%%%%%%%%%%%%%%%%%%

\subsection{Folds \& Stallings fold decompositions}{\label{ss:StallingsFoldDecompositions}}

%%%%%%%%%%%%%%%%%%%%%%%%%%%%%%%%%%%%%%%%%%%%%%%%%%%%%%%%%%%%%

Suppose $\G$ and $\G'$ are graphs viewed topologically and $e_0,e_1\in E^{\pm}\G$ are distinct directed edges emanating from a common vertex. Then $\G'$ is obtained from $\G$ by a \emph{proper full fold of $e_1$ over $e_0$} when there exist orientation-preserving homeomorphisms $\sigma_0\from [0,1]\to e_0$ and $\sigma_1\from [0,2]\to e_1$ so that 
$\G'=\G\backslash\sim$ is the topological quotient of $\G$ with respect to the equivalence relation $\sim$ defined by $\sigma_0(t)=\sigma_1(t)$ for each $t\in [0,1]$. Further, $\G'$ is obtained from $\G$ by a \emph{complete fold of $e_0$ and $e_1$} if instead $\sigma_1\from [0,1]\to e_1$ and a \emph{partial fold of $e_0$ and $e_1$} if instead $\sigma_0\from [0,2]\to e_0$. Proper full folds, complete folds, and partial folds are together called \emph{folds}.

\vspace*{-1.5mm}

\parpic[r]{\includegraphics[width=3in]{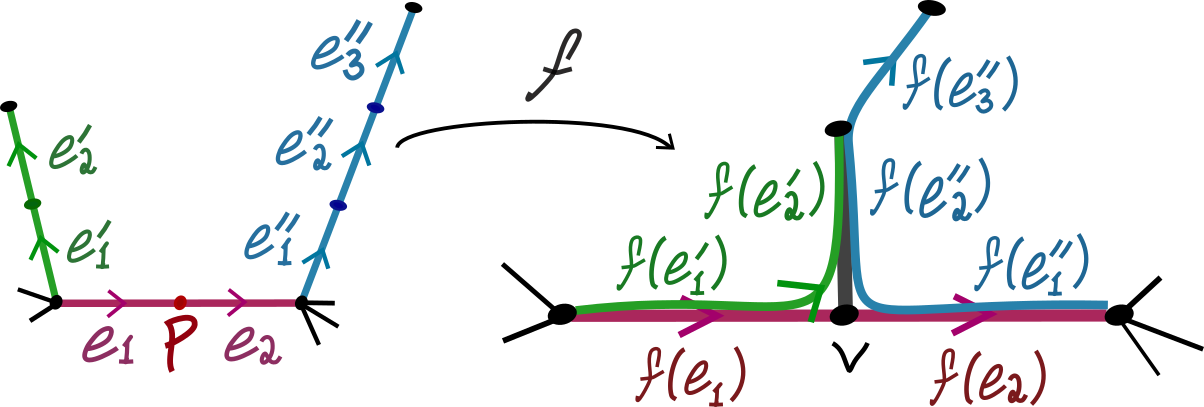}} 
Suppose there are edges $e,e',e''\in E\G$ such that a partial fold of $\{e,e'\}$ and a partial fold of $\{\bar{e},e''\}$ terminate at a common point $p$ in $e$. Suppose further that we can write $e=e_1 e_2$, and $e'=e_1' e_2'$, and $e''=e_1'' e_2'' e_3''$ so that $f(e_1) = f(e_1')$, and $f(\overline{e_2}) = f(e_1'')$, and $f(e_2') = f(e_2'')$. Define a \emph{tripod proper full fold} by first folding $\{e,e'\}$ and $\{\bar{e},e''\}$, then afterward folding $\{e_2',e_2''\}$. A tripod fold could be seen as the composition of three folds, namely two partial folds and a proper full fold.

Note that a proper full fold does not change the number of edges, a complete fold decreases the number of edges, and a partial fold or tripod proper full fold increases the number of edges.

\vspace*{-1.5mm}

%%%%%%%%%%%%%%%%%%%%%%%%%%%%%%%%%%%%%%%%%%%%%%%%%%%%%%%%%%%%%

\subsubsection{Proper full fold notational conventions}{\label{ss:PffNotation}}

In the automata defined here, all folds are proper  
\vspace*{-1.5mm}
\parpic[l]{\includegraphics[width=2.1in]{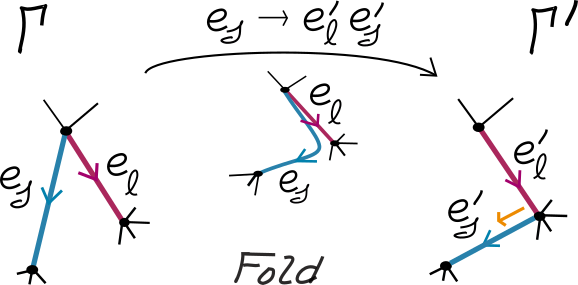}} 
\noindent full folds. We thus establish here notational conventions for proper full folds. Suppose $f\from\G\to\G'$ is a single proper full fold of an edge $e_j$ over an edge $e_{\ell}$, as depicted to the left. Apart from $e_j$, each $e_i\in E\G$ is mapped to a single edge of $\G'$, which we call $e_i'$. The image of $e_j$ is an edge-path in $\G'$ with 2 edges, the latter of which we call $e_j'$. The map $f$ is then defined by 

\vspace*{-4mm}

\begin{equation}\label{eq:f}
f \colon
\begin{cases}
e_j\mapsto e_{\ell}' e_j' \\
e_i \mapsto e_i' \text{ for } i\neq 1\\
\end{cases}
\end{equation}

\vspace*{-2mm}

\noindent and for such a map we just write $f\from e_j\mapsto e_{\ell}' e_j'$, or even abusively $e_j\mapsto e_{\ell}e_j$. We will assume it understood that $f(\ol{e_j})=\ol{e_j'}\ol{e_{\ell}'}$. Call the edge-labeling of $\G'$ just described the \emph{induced edge-labeling}. For the map induced by folding $\ol{e_{\ell}}$ over $\ol{e_j}$, we may write $f \from e_j\mapsto e_j e_{\ell}$.

The direction map $Df$ for $f$ will be

\vspace*{-5mm}

\begin{equation}\label{eq:Df}
Df(e_i) =
\begin{cases}
 e_{\ell}' \text{ for } i = j \\
e_i' \text{ for } i\neq j\\
\end{cases}
\end{equation}

\noindent and the direction indicated by an orange arrow in the figure above, namely $e_j'$, is not in the image.

%%%%%%%%%%%%%%%%%%%%%%%%%%%%%%%%%%%%%%%%%%%%%%%%%%%%%%%%%%%%%

\subsubsection{Stallings fold decompositions}{\label{ss:PffNotation}}

We follow the \cite{wiggd1} description and language of Stallings fold decompositions. In general, each fully irreducible $\vphi\in\out$ has multiple tt representatives, each of which can have several distinct Stallings fold decompositions. An example of the multiple Stallings fold decompositions arising from a tt map with 2 illegal turns is given in Example \ref{ex:StallingsFolds}.

Since a Stallings fold decomposition can end with a homeomorphism changing edge labels, we define an \emph{edge-permutation graph isomorphism} as a graph isomorphism that possibly permutes edge labels and possibly reverses the orientations on some subset of the edges.

\vskip25pt

%%%%%%%%%%%%%%%%%%%%%%%%%%%%%%%%%%%%%%%%%%%%%%%%%%%%%%%%%%%%%%%%%%%%%%%%%%%%%%%%%%%
%%%%%%%%%%%%%%%%%%%%%%%%%%%%%%%%%%%%%%%%%%%%%%%%%%%%%%%%%%%%%%%%%%%%%%%%%%%%%%%%%%

\section{Indices, the index inequality, \& the index defecit}{\label{ss:indices}}

%%%%%%%%%%%%%%%%%%%%%%%%%%%%%%%%%%%%%%%%%%%%%%%%%%%%%%%%%%%%%

\subsection{Index sum \& list}{\label{ss:indexsum}}

Let $\vphi \in Out(F_r)$ be fully irreducible. For each component $C_i$ of $IW(\vphi)$, let $k_i$ denote the number of vertices of $C_i$. Then the \emph{index sum} (also called the \emph{rotationless index}) is defined as
$i(\vphi) := \sum 1-\frac{k_i}{2}$. Since the index sum can be computed as such from an ideal Whitehead graph, one can define an index sum for an ideal Whitehead graph, or in fact for any finite graph. For a graph $\mG$, we denote the index sum by $i(\mG)$. Note that the index sum is always in $\frac{1}{2}\ZZ$.

Writing the terms $1-\frac{k_i}{2}$ as a list, we obtain the \emph{index list} for $\vphi$. By \cite{gjll}, we know that each fully irreducible $\varphi \in Out(F_r)$ satisfies the ``index inequality'':
\begin{equation}\label{eq:indexinequality}
0 > i(\varphi) \geq 1-r,\quad \text{ or equivalently, }\quad \frac{1}{2} \geq i(\varphi) \geq 1-r.
\end{equation}
It is constructively shown in \cite{c15} that, for each rank $r\geq 4$ and list of negative $\frac{1}{2}$-integers (hence sum) satisfying Equation \ref{eq:indexinequality}, there is a fully irreducible outer automorphism in $\out$ having that list as its index list. The case of $r=3$ was proved in \cite{p15}.

%%%%%%%%%%%%%%%%%%%%%%%%%%%%%%%%%%%%%%%%%%%%%%%%%%%%%%%%%%%%%

\subsection{Ageometric fully irreducible outer automorphisms}{\label{ss:indexsum}}

An ageometric fully irreducible $\varphi \in Out(F_r)$ can be characterized by satisfying $0 > i(\varphi) > 1-r$, a phenomena proved generic in \cite{randomout}. The index list, index sum, and ideal Whitehead graph are all invariant under taking powers of the outer automorphism, hence the same holds for being ageometric.

By \cite{bh92}, all fully irreducible outer automorphisms have ``stable'' tt representatives and stable tt representatives of ageometric outer automorphisms have no NPs \cite[Theorem 3.2]{bf94}. They in fact have tt representatives for which all powers are stable. As such, we see no loss in assuming throughout that all tt representatives of ageometric outer automorphisms have no PNPs, thus also satisfy that their ideal Whitehead graph is the disjoint union of their stable Whitehead graphs.

Fully irreducible outer automorphisms that are not ageometric can be geometric, i.e. induced by surface homeomorphisms, or parageometric. We avoid both these circumstances, as they are rare.

\subsubsection{Ageometric full irreducibility criterion}{\label{ss:fic}}

We use the criterion of Proposition \ref{prop:FIC} to prove that certain tt maps represent ageometric fully irreducible outer automorphisms. Proposition \ref{prop:FIC} is \cite[Proposition 3.35]{stablestrata}, which is the elevation of \cite[Proposition 4.1]{IWGII} to include the observation that a fully irreducible with a PNP-free tt representative is in fact ageometric. 

\begin{prop}[\cite{IWGII}, \cite{stablestrata}](\emph{The Ageometric Full Irreducibility Criterion (FIC)})\label{prop:FIC}
Let $g\colon \Gamma \to \Gamma$ be a PNP-free, irreducible tt representative of $\vphi \in Out(F_r)$. Suppose that the transition matrix for $g$ is Perron-Frobenius and that the local Whitehead graph at each vertex of $\G$ is connected. Then $\vphi$ is an ageometric fully irreducible outer automorphism.
\end{prop}

\vskip10pt

\subsection{Index deficit}{\label{ss:id}}

In light of Equation \ref{eq:indexinequality}, and following language proposed by Lee Mosher in conversations, we call 
\begin{equation}\label{eq:id}
ID(\varphi) = i(\varphi) + r - 1
\end{equation}  
the \emph{index deficit} of $\vphi$. For any tt representative $g\from\G\to\G$ of $\vphi$, the Euler characteristic $\chi(\G)$ of $\G$ satisfies $\chi(\G) = 1-r$. So, in fact,
\begin{equation}\label{eq:indexdefecit}
ID(\varphi) = i(\varphi) - \chi(\G).
\end{equation} 

As an ageometric fully irreducible $\varphi \in Out(F_r)$ is characterized by satisfying $0 > i(\varphi) > 1-r$, it could also be characterized by having positive index deficit. The follow lemma says that $ID(\vphi)$ ranges from 0 (for geometric and parageometric $\varphi \in Out(F_r)$) to $r-\frac{1}{2}$:\\

\begin{lem}[Index deficit values]\label{lemma:IndexDefecitsAchieved} For each integer $r\geq 3$, each $\frac{1}{2}$-integer value $R$ satisfying 
\begin{equation}\label{eq:indexdefecit}
0 \leq R \leq r-\frac{1}{2}
\end{equation}
is realized as the index deficit $ID(\vphi)$ for a fully irreducible $\vphi\in\out$.
\end{lem}

\begin{proof}
Suppose that $r\geq 3$ and $R$ is a $\frac{1}{2}$-integer value satisfying $0 \leq R \leq r-\frac{1}{2}$. Then $R-r+1$ is a $\frac{1}{2}$-integer value satisfying $\frac{1}{2} \geq R-r+1 \geq 1-r$.
Thus, by \cite{c15}, there exists a fully irreducible $\vphi\in\out$ so that $i(\varphi)=R-r+1$, i.e. $R=i(\varphi)+r-1$. So $R=ID(\varphi)$, as desired.
\qedhere
\end{proof}

\vskip25pt

%%%%%%%%%%%%%%%%%%%%%%%%%%%%%%%%%%%%%%%%%%%%%%%%%%%%%%%%%%%%%%%%%%%%
%%%%%%%%%%%%%%%%%%%%%%%%%%%%%%%%%%%%%%%%%%%%%%%%%%%%%%%%%%%%%%%%%%%%

\section{Outer space $\os$, fully irreducible axes, \& lone axis outer automorphisms}{\label{s:axes}}

%%%%%%%%%%%%%%%%%%%%%%%%%%%%%%%%%%%%%%%%%%%%%%%%%%%%%%%%%%%%%

%%%%%%%%%%%%%%%%%%%%%%%%%%%%%%%%%%%%%%%%%%%%%%%%%%%%%%%%%%%%%

\subsection{Outer space $\os$}{\label{s:OS}}

%%%%%%%%%%%%%%%%%%%%%%%%%%%%%%%%%%%%%%%%%%%%%%%%%%%%%%%%%%%%%%%%%%%%

Outer space $\os$ was first defined in \cite{cv86}. We do not use in this manuscript all details of the definitions, so only hit on a few highlights here and then refer the reader to \cite{FrancavigliaMartino,b15,v15} for further reading on the topic.

Points in $\os$ are triples $(\G,m,\ell)$, called \emph{marked metric graphs}, where\\
$\bullet$ $\G$ is a finite graph such that valence$(v)\geq 3$ for each $v\in V\G$, and\\
$\bullet$ $m\colon F_r\to \pi_1(\G)$ is an isomorphism, called a \emph{marking}, and\\
$\bullet$ $\ell\from E\G\to\RR_+$ is an assignment of \emph{lengths} to edges satisfying that $\sum_{e\in E\G}\ell(e)=1$.

\noindent Two triples are equivalent that differ by an isometric change of marking.

Outer space can be endowed with what is known as the \emph{Lipschitz metric} $\mL$, which is in fact not a metric, as it is asymmetric. $\out$ acts on $\os$ isometrically by changing the marking.

%%%%%%%%%%%%%%%%%%%%%%%%%%%%%%%%%%%%%%%%%%%%%%%%%%%%%%%%%%%%%

\subsection{Fold line geodesics in $\os$}{\label{s:foldlinegeodesics}}

%%%%%%%%%%%%%%%%%%%%%%%%%%%%%%%%%%%%%%%%%%%%%%%%%%%%%%%%%%%%%%%%%%%%

In \cite{s89}, Skora interpreted a Stallings fold decomposition for a graph map  homotopy equivalence $g\colon \Gamma \to \Gamma'$ as a sequence of folds performed continuously. In \cite[Proposition 3.17]{DenseGeodesic} it is proved that any fold sequence similarly determines a geodesic in $\os$ provided that there is some conjugacy class in $F_r$ whose realization in the graphs of the sequence is never folded. We give here a version specialized for our purposes:

\vskip10pt

\begin{prop}\cite[Proposition 3.17]{DenseGeodesic}\label{foldLineGeodesics}
Let $\{\mF_i \from x_i \to x_{i+1}\}_{i=0}^k$ be a sequence of folds in $\os$ and suppose there is a conjugacy class $\al$ in $F_r$ satisfying that, for each $i$, the realization $\al_{x_i}$ of $\alpha$ in $x_i$ is legal with respect to $\mF_i$, ie. is not folded by $\mF_i$. Then the corresponding fold path $\im (\mF) = \{ x_t \}_{t \in [0,k]}$ is an unparametrized geodesic, i.e. for each $r\leq s \leq t$ in $[0,k]$, we have $d(x_r,x_t) = d(x_r, x_s) + d(x_s, x_t)$.
\end{prop}

\vskip10pt

%%%%%%%%%%%%%%%%%%%%%%%%%%%%%%%%%%%%%%%%%%%%%%%%%%%%%%%%%%%%%%%%%%%%

\subsection{Axes in outer space}{\label{ss:something}}

 Let $g\colon \Gamma \to \Gamma$ be an expanding irreducible tt map representing $\vphi \in Out(F_r)$ and let $\lambda>1$ be its PF eigenvalue. Suppose further that $g_k\circ \cdots \circ g_1$ is a Stallings fold decomposition of $g$. Repeating the decomposition defines a periodic fold line in $\os$. A discretization of this fold line is depicted in Equation \ref{E:PeriodicFoldLines} below, where $\G_0$ depicts $\G$ endowed with the metric determined by the PF eigenvector and $\Gamma_{nK}=\frac{1}{\lambda^n}\Gamma_0 \cdot \varphi^n$ for each $n\in\ZZ$.
\vspace*{-1mm}
\begin{equation}\label{E:PeriodicFoldLines}
\dots \xrightarrow{} \Gamma_0 \xrightarrow{g_1} \Gamma_1 \xrightarrow{g_2} \cdots \xrightarrow{g_K} \Gamma_K \xrightarrow{g_{K+1}} \Gamma_{K+1} \xrightarrow{g_{K+2}} \cdots \xrightarrow{g_{2K}} \Gamma_{2K} \xrightarrow{g_{2K+1}} \dots
\end{equation}

The process of Skora defines a path $\mL_0 \from [0,\log\lam] \to \os$ so that the union of $\varphi^k$-translates of $\mL_0$ for all $k$ gives the entire fold line $\mL$. That is, $\mL \from \RR \to \os$ is defined by $\mL(t) = \mL_0(t - \lfloor \frac{t}{\log\lam}\rfloor) \varphi^{\lfloor \frac{t}{\log\lam}\rfloor}$.
$\mL$ is called a \emph{periodic fold line} for $\vphi$ or, if $\vphi$ is fully irreducible, an \emph{axis} for $\vphi$.

\cite[Lemma 2.7]{stablestrata} implies that the periodic fold lines determined by tt representatives of fully irreducible outer automorphisms are Lipschitz geodesics.

%%%%%%%%%%%%%%%%%%%%%%%%%%%%%%%%%%%%%%%%%%%%%%%%%%%%%%%%%%%%%%%%%%%%
%%%%%%%%%%%%%%%%%%%%%%%%%%%%%%%%%%%%%%%%%%%%%%%%%%%%%%%%%%%%%%%%%%%%

\subsubsection{Fold-conjugate decompositions}{\label{ss:FoldConjugate}}

%%%%%%%%%%%%%%%%%%%%%%%%%%%%%%%%%%%%%%%%%%%%%%%%%%%%%%%%%%%%%%%%%%%%

Since an axis for a fully irreducible $\vphi\in\out$ has a
\parpic[r]{\includegraphics[width=.75in]{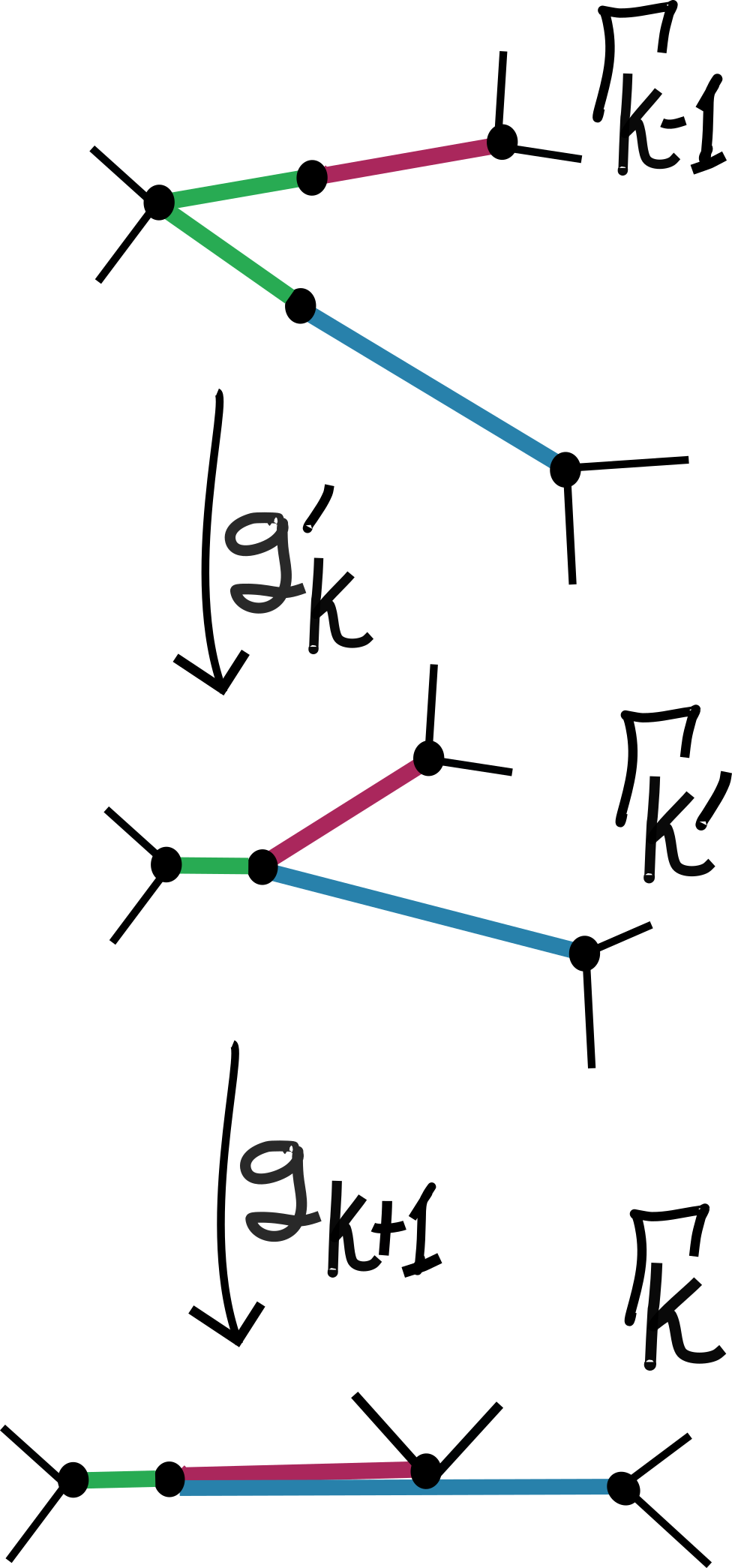}}  \noindent periodic structure, one can view its Stallings fold decompositions cyclically. With careful attention paid to the marking, one can see that starting at a different fold in a decomposition now yields a tt map representing the same outer automorphism. Further, tt representatives may start ``in the middle of a fold.'' These notions of cyclically permuting a Stallings fold decomposition or, equivalently, shifting along an axis are formalized in  \cite{wiggd1} via the language of \emph{fold-conjugate} and \emph{partial-fold conjugate} decompositions. We use the definitions presented there, just including here that a \emph{subdivided fold} is a fold written as a composition of two folds, as depicted to the right. Since fold-conjugate tt maps represent the same outer automorphisms, they share all outer automorphism invariants, such as ideal Whitehead graphs, indices, and whether or not a map is fully irreducible.

\subsection{Lone axis outer automorphisms}{\label{s:loneaxes}}

A main focus of this manuscript is a class of ageometric fully irreducible outer automorphisms proved in \cite[Theorem 4.7]{loneaxes} to have a unique axis in $\os$:

A \emph{lone axis fully irreducible outer automorphisms} is an ageometric fully irreducible outer automorphism $\varphi \in Out(F_r)$ satisfying that ~\\
\vspace{-5mm}
\begin{enumerate}
\item the rotationless index satisfies $i(\varphi) = \frac{3}{2}-r$ and 
\item no component of the ideal Whitehead graph $IW(\varphi)$ has a cut vertex.
\end{enumerate}

Each tt representative of each lone axis fully irreducible $\vphi$ is PNP-free (\cite[Lemma 4.4]{loneaxes}) and has a unique Stallings fold decomposition (\cite[Theorem 4.7]{loneaxes}). Stallings fold decompositions of tt representatives of the same lone axis fully irreducible outer automorphism are fold-conjugate as, by \cite[Theorem 4.7]{loneaxes}, they determine the same axis in $\os$.

The following proposition, which is \cite[Corollary 3.8]{loneaxes}, is inspiration for what we will define in \S \ref{ss:MaximallySingular} as ``fully singular outer automorphisms.''

\begin{prop}[\cite{loneaxes}]\label{P:EveryVertexPrincipal}
Let $\vphi$ be a lone axis fully irreducible outer automorphism. Then there exists a train track representative $g \from \Gamma \to \Gamma$ of some power $\vphi^R$ of $\vphi$ so that all vertices of $\Gamma$ are principal, and fixed, and all but one direction is fixed.
\end{prop}

\vskip25pt

%%%%%%%%%%%%%%%%%%%%%%%%%%%%%%%%%%%%%%%%%%%%%%%%%%%%%%%%%%%%%
%%%%%%%%%%%%%%%%%%%%%%%%%%%%%%%%%%%%%%%%%%%%%%%%%%%%%%%%%%%%%

\section{Proper full fold (pff) decompositions \& standard notation}{\label{ss:PffDecompNotation}}

Suppose one has a Stallings fold decomposition $\G_{0} \xrightarrow{g_1} \G_{1} \xrightarrow{g_2} \cdots \xrightarrow{g_{n-1}} \G_{n-1} \xrightarrow{g_n} \G_n$ of a homotopy  
\vspace*{-3mm}  
\parpic[r]{\includegraphics[width=2.1in]{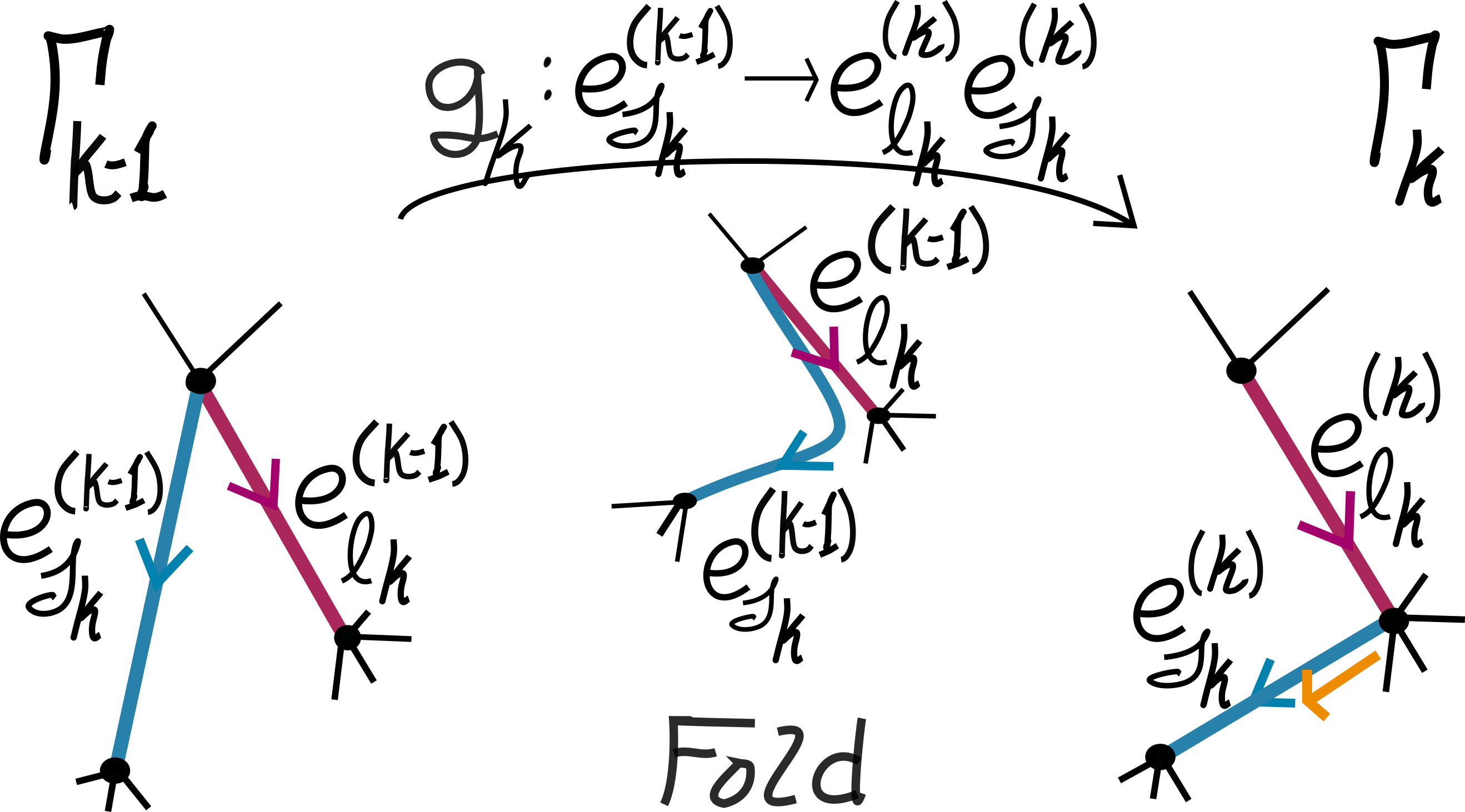}}
\noindent
 equivalence tt map $g \from \Gamma \to \Gamma$, where for each $1\leq k < n$, the fold $g_k \from \Gamma_{k-1} \to \Gamma_k$ is a proper full fold of an edge $e_{j_{k}}^{(k-1)}$ over an edge $e_{\ell_{k}}^{(k-1)}$, and $g_n$ is a graph isomorphism, possibly changing edge labels and orientations. We call such a fold decomposition (and its corresponding fold line in $\os$) a \emph{proper full fold (pff) decomposition/line}. For such a pff decomposition we use the following notation for $1\leq k <n$ and call it the \emph{standard notation}:
\vspace{-2mm}
\begin{equation}\label{eq:g_k}
g_k \colon
\begin{cases}
e_{j_{k}}^{(k-1)}\mapsto e_{\ell_{k}}^{(k)}e_{j_{k}}^{(k)} \\
\\
e_i^{(k-1)} \mapsto e_i^{(k)} \quad \text{ for } i\neq j_{k}\\
\end{cases}
\end{equation}

\vskip3pt

\noindent We write $g_k\from e_{j_{k}}^{(k-1)}\mapsto e_{\ell_{k}}^{(k)}e_{j_{k}}^{(k)}$, or even more abusively $e_{j_{k}}\mapsto e_{\ell_{k}}e_{j_{k}}$. The direction map for $g_k$ is
\vspace{-0.5mm}
\begin{equation}\label{eq:Dg_k}
Dg_k \colon
\begin{cases}
e_{j_{k}}^{(k-1)}\mapsto e_{\ell_{k}}^{(k)} \\
\\
e_i^{(k-1)} \mapsto e_i^{(k)} \quad \text{ for } i\neq j_{k}\\
\\
\overline{e_i^{(k-1)}} \mapsto \overline{e_i^{(k)}} \quad \text{ for each } i\\
\end{cases}
\end{equation}

\vskip5pt

The map $g_n$ may be a homeomorphism. In such a case, $g_n$ is a bijection on edges (possibly reversing some orientations) and $Dg_n$ is a bijection on directions satisfying $Dg_n(\ol{e}) = \ol{Dg_n(e)}$ for each $e\in E\G_{n-1}$. Note that the direction map for $g$ is always  $Dg=Dg_n\circ\cdots\circ Dg_1.$

For brevity we use the notation:
\begin{equation}\label{eq:g_ij}
g_{j,i}:=
\begin{cases}
g_j\circ\cdots\circ g_i \quad \quad \quad \quad \quad \quad \quad \text{ if } i\leq j \\
g_j\circ\cdots\circ g_1 \circ g_n\circ\cdots\circ g_i \quad \text{ if } j<i \\
\end{cases}
\end{equation}

\vskip5pt
\noindent Throughout this manuscript we reserve the notation $f_k$ for 
$g_{k,1}\circ g_{n,k+1}\from \G_k\to\G_k$.

\vskip15pt

\begin{lem}\label{lem:graphmap}
Suppose $\G_{0} \xrightarrow{g_1} \G_{1} \xrightarrow{g_2} \cdots \xrightarrow{g_{n-1}} \G_{n-1} \xrightarrow{g_n} \G_n$ is a Stallings fold decompositions of an expanding irreducible homotopy equivalence tt map $g \from \Gamma \to \Gamma$, where each fold $g_k \from \Gamma_{k-1} \to \Gamma_k$ is a proper full fold defined by $g_k\from e_{j_{k}}^{(k-1)}\mapsto e_{\ell_{k}}^{(k)}e_{j_{k}}^{(k)}$. According to the standard notation:
\begin{enumerate}[(a)]
\item the direction $e_{j_{n}}$ is not in the image of $Dg$, and
\item the turn $\{\overline{e_{\ell_n}},e_{j_{n}}\}$ is a $g$-taken turn, and
\item 
$$\tau (g_n\circ\cdots\circ g_1)=\tau (g_n)\bigcup_{k=1}^{n-1}Dg_{n,k+1}(\tau(g_k))\text{, and}$$
\item repeating the Stallings fold decomposition if necessary,
the illegal turns for $g$ are $\{e_{j_1}, e_{\ell_1}\}$ and each nondegenerate turn $\{d_1, d_2\}$ such that, for some $k$,
$$Dg_{k-1,1}(d_1)=e_{j_k} \quad and \quad Dg_{k-1,1}(d_2)=e_{\ell_k}.$$    
\end{enumerate}
Assume $g_n$ is instead an edge-permutation graph isomorphism, as in a pff decomposition. Then (a), (b), and (d) are replaced as follows with (a $'$), (b $'$), and (d $'$).
\begin{enumerate}
\item[(a $'$)] $Dg_n(e_{j_{n-1}})\notin Image(Dg)$, and
\item[(b $'$)] the turn $\{Dg_n(\overline{e_{\ell_{n-1}}}),Dg_n(e_{j_{n-1}})\}$ is a $g$-taken turn, and
\item[(d $'$)] the statement of (d) holds but with the added requirement that $k\neq pn$ for any $p\in\ZZ_{\geq 1}$. 
\end{enumerate}
\end{lem}

\begin{proof} 
a) $e_{j_{n}}$ is not in the image of $Dg_n$ and $Dg=Dg_n\circ\cdots\circ Dg_1.$\\
b) Proper full folds are surjective, so each $g_k$ is surjective, so $g_{n-1,1}$ is surjective. Thus, $e_{j_{n-1}}$ is in the $g_{n-1,1}$-image of some $e\in E^{\pm}\G$ and so $g(e)$ contains $e_{\ell_n} e_{j_n}$. Thus, $\{\overline{e_{\ell_n}},e_{j_n}\}$ is a $g$-taken turn, as desired.\\
c) Since each $g_k$ is surjective and the standard proper full fold notation is consistent with that on a rose, the same proof as that of \cite[Lemma 2.12]{counting} works here also.\\
d) We first show that $\{e_{j_1}, e_{\ell_1}\}$ is an illegal turn, as is each turn $\{e_{j_{k}}, e_{\ell_{k}}\}$ such that both $e_{j_{k}}$ and $e_{\ell_{k}}$ are in the image of $Dg_{k-1,1}$. Since $g_1\from\G_0\to\G_1$ is defined by $e_{j_1}\mapsto e_{\ell_{1}} e_{j_{1}}$, we have
$$\{Dg(e_{j_1}),~Dg(e_{\ell_1})\}=
\{Dg_{k,2}(Dg_1(e_{j_1})),~ Dg_{k,2}(Dg_1(e_{\ell_1}))\}=\{Dg_{k,2}(e_{\ell_1}), ~Dg_{k,2}(e_{\ell_1})\},$$
\noindent so is degenerate. And $\{e_{j_1}, e_{\ell_1}\}$ is an illegal turn, as desired. 

Suppose $\{d_1, d_2\}$ satisfies $Dg_{k-1,1}(d_1)=e_{j_k}$ and $Dg_{k-1,1}(d_2)=e_{\ell_k}$ for some $k\in\ZZ_{>1}$. Then 
$$\{Dg(d_1),~Dg(d_2)\}=
\{Dg_{n,k}(Dg_{k-1,1}(d_1)), ~Dg_{n,k}(Dg_{k-1,1}(d_2))\}=
\{Dg_{n,k}(e_{j_k}), ~Dg_{n,k}(e_{\ell_k})\}$$
$$=\{Dg_{n,k+1}(Dg_{k}(e_{j_k})),~Dg_{n,k+1}(Dg_{k}(e_{\ell_k}))\}=
\{Dg_{n,k+1}(e_{j_k}),~Dg_{n,k+1}(e_{j_k})\},$$
\noindent so is also degenerate. And $\{d_1, d_2\}$ is an illegal turn, as desired. 

We now show that $\{e_{j_1}, e_{\ell_1}\}$ and such $\{d_1, d_2\}$ are the only illegal turns. Each $Dg_k$ is bijective on directions except for identifying $e_{j_k}$ and $e_{\ell_k}$. Thus, $Dg_k$ can only identify two directions in $Image(Dg_{k-1,1})$ if they are $e_{j_k}$ and $e_{\ell_k}$.  Since $Dg=Dg_n\circ\cdots\circ Dg_1$, we cannot have any  $Dg^m(d_1)=Dg^m(d_2)$, unless $D(g_{k-1,1})(d_1)=e_{j_k}$ and $D(g_{k-1,1})(d_2)=e_{\ell_k}$, or vice versa, for some $k$.

Now assume $g_n$ is instead an edge-permutation graph isomorphism.\\
a $'$) $e_{j_{n-1}}\notin Image(Dg_{n-1})$, and $Dg_n$ is bijective, and $Image(Dg)=Image(Dg_{n}\circ Dg_{n-1}\circ\cdots\circ Dg_1)$.\\
b $'$) By (b), we have $g_{n-1,1}(e) = \dots e_{\ell_{n-1}} e_{j_{n-1}}\dots$ for some $e\in E^{\pm}\G$ and so 
$g(e)=g_n(g_{n-1,1}(e)) = \dots g_n(e_{\ell_{n-1}}) g_n(e_{j_{n-1}})\dots$ and takes the turn $\{\overline{g_n(e_{\ell_{n-1}})}, g_n(e_{j_{n-1}})\} = \{Dg_n(\overline{e_{\ell_{n-1}}}), Dg_n(e_{j_{n-1}})\}$. \\
d $'$) $Dg_n$ is a bijection, so does not identify directions, but the remainder of the proof of (d) holds.
\\
\qedhere
\end{proof}

\vskip25pt

%%%%%%%%%%%%%%%%%%%%%%%%%%%%%%%%%%%%%%%%%%%%%%%%%%%%%%%%%%%%%%%%%%%%

\section{Abstract lamination train track (ltt) structures}{\label{s:AbstractLttStructures}}

%%%%%%%%%%%%%%%%%%%%%%%%%%%%%%%%%%%%%%%%%%%%%%%%%%%%%%%%%%%%%%%%%%%%

In preparation to discuss tt automata, we need an abstract notion of an ltt structure. Suppose $\G$ is a directed finite graph of first betti number $r$ and such that the valence of each vertex is $\geq 3$. An \emph{abstract lamination train track (ltt) structure} $\mG$ with underlying graph $\G$ and ``index'' $\mI$ is a partially colored partially directed labeled graph with:
\begin{itemize}
  \item[(ltt-i)] a vertex for each direction in $\mG$, colored either red or purple, and
  \item[(ltt-ii)] $2(\mI-\chi(\G))$, i.e. $2(\mI +r-1)$, of the vertices are colored red, and
  \item[(ltt-iii)] at least one red vertex contained in precisely one red edge, and
\item[(ltt-iv)] for each edge $e\in E\G$, a directed black edge labeled with $e$ directed from the vertex labeled $e$ to the vertex labeled $\bar{e}$, and
  \item[(ltt-v)] undirected purple edges connecting some portion of the pairs of purple vertices forming turns in $\G$ (i.e. that represent distinct directions at a common vertex in $\G$) and so that each vertex is contained in at least one colored (purple or red) edge and there are never 2 colored edges connecting the same pair of vertices, and
  \item[(ltt-vi)] for each $v\in V\G$, the subgraph of $\mG$ that is the union of the colored (purple or red) edges representing turns at $v$, which we denote by $LW(\mG,v)$, is connected.
\end{itemize}
\vspace{-1mm}
\noindent We denote the edge connecting the turn $\{d_1,d_2\}$ by $[d_1,d_2]$, recognizing that $[d_1,d_2]=[d_2,d_1]$.

\vspace{2mm}
The following language and notation will be used both for abstract ltt structures and those defined by tt maps.

As mentioned in (ltt-vi), $LW(\mG,v)$, called the \emph{local Whitehead graph at $v$}, denotes the subgraph of $\mG$ that is the union of the colored edges representing turns at $v$. We let $SW(\mG,v)$ denote the purple subgraph of $\mG$ and call it the \emph{stable Whitehead graph at $v$}. The \emph{ideal Whitehead graph of $\mG$} is $IW(\mG):=\bigsqcup SW(\mG,v)$, as $v$ varies over $V\G$.
Note that, if $g$ is a tt representative of a fully irreducible $\vphi$ satisfying that $g$ has PNPs and $ltt(g)=\mG$, then $IW(\vphi)\neq IW(\mG)$.
% We note that, in the presence of PNPs, this definition of the ideal Whitehead graph would not match $IW(\vphi)$ for a fully irreducible with .

We consider a path in an ltt structure $\mG$ \emph{smooth} that alternates between black and colored edges. An ltt structure is \emph{birecurrent} that contains a bi-infinite smooth path $\gamma$ so that $\gamma~\backslash~\sigma$ contains each edge of $\mG$ for each finite subpath $\sigma$ of $\gamma$.

\vskip25pt

%%%%%%%%%%%%%%%%%%%%%%%%%%%%%%%%%%%%%%%%%%%%%%%%%%%%%%%%%%%%%%%%%%%%%%%%%%%%%%%%%%%
%%%%%%%%%%%%%%%%%%%%%%%%%%%%%%%%%%%%%%%%%%%%%%%%%%%%%%%%%%%%%%%%%%%%%%%%%%%%%%%%%%

\section{Fully singular outer automorphisms \& train track representatives}\label{ss:MaximallySingular}

%%%%%%%%%%%%%%%%%%%%%%%%%%%%%%%%%%%%%%%%%%%%%%%%%%%%%%%%%%%%%%%%%%%%%%%%%%%%%%%%%%

Suppose that $g\from\G\to\G$ is a train track representative of an ageometric fully irreducible $\vphi\in\out$.
As in \cite{hm11} and \cite{loneaxes}, we call a periodic point in $\G$ ($g$-)\emph{principal} if it is either the endpoint of a PNP or has $\geq 3$ periodic directions. Thus, if $g$ has no PNPs, the principal vertices are precisely those periodic vertices with $\geq 3$ periodic directions. For a set $g$, let $PV\G$ denote the set of $g$-principal vertices of $\G$ and $nPV\G$ the set of $g$-nonprincipal vertices of $\G$.

\vskip10pt

%%%%%%%%%%%%%%%%%%%%%%%%%%%%%%%%%%%%%%%%%%%%%%%%%%%%%%%%%%%%%%%%%%%%
%%%%%%%%%%%%%%%%%%%%%%%%%%%%%%%%%%%%%%%%%%%%%%%%%%%%%%%%%%%%%%%%%%%%

\subsection{Fully singular \& fully preprincipal tt representatives}{\label{ss:MaxSing}}

%%%%%%%%%%%%%%%%%%%%%%%%%%%%%%%%%%%%%%%%%%%%%%%%%%%%%%%%%%%%%%%%%%%%

We call a tt representative $g$ of an ageometric fully irreducible $\vphi\in\out$ \emph{fully singular} if it is PNP-free and each vertex is principal, i.e. if $PV\G=V\G$. Note that in such a case, the vertices of $\G$ are in 1-to-1 correspondence with the components of the ideal Whitehead graph.  

We call a tt representative $g$ of a fully irreducible $\vphi\in\out$ \emph{fully preprincipal} if each vertex has $\geq 3$ gates, additionally requiring that $g$ is PNP-free if $\vphi$ is ageometric. Thus, each fully singular representative is fully preprincipal, but the converse does not hold (even for PNP-free tt representatives). 

It is unclear at this point whether each ageometric fully irreducible outer automorphism has a fully singular tt representative. However, in the case where $i(\vphi)=\frac{3}{2}-r$, such as for lone axis outer automorphisms, \cite[Corollary 3.8]{loneaxes} implies the existence of fully singular tt representatives. 

We now prove the existence of fully preprincipal tt representatives.

\begin{prop}[Fully preprincipal tt representatives]\label{prop:Preprincipal}
Each fully irreducible $\vphi\in\out$ has a fully preprincipal tt representative. 
%If $\vphi$ is ageometric, then the representative can be chosen to be PNP-free. 
\end{prop}

\begin{proof} 
Suppose $\vphi\in\out$ is fully irreducible. We show $\vphi$ has a (PNP-free if $\vphi$ is ageometric) tt representative for which each vertex has at least three gates. Since $\vphi$ is fully irreducible, it has a tt representative. In the case where $\vphi$ is ageometric, the representative can be chosen to additionally be PNP-free. Let $h\from\G\to\G$ be one such representative. We cannot a priori assume $h$ is fully preprincipal, so we suppose some $v\in V\G$ has less than three $h$-gates. If $v$ had only one gate, then any edge passing through $v$ would have backtracking, contradicting that $h$ is a tt representative of a fully irreducible outer automorphism (an invariant graph could be inserted at $v$ if no edge passed over $v$). If $v$ had 2 gates, we could perform folds, similar to those in \cite[pg. 16-17]{bh92}, to obtain a new (PNP-free if $\vphi$ was ageometric) tt representative of $\vphi$ with fewer vertices with less than three gates:

One of the gates must have multiple directions in it, or $v$ would have valence two. We fold that gate $\lambda$. Suppose that $k$ is the power of $h$ that identifies all of the directions $E_1,\dots, E_n$ of $\lambda$. Maximally fold $h^{k-1}(E_1),\dots, h^{k-1}(E_n)$. We are then able to fold $h^{k-2}(E_1),\dots, h^{k-2}(E_n)$ and do so maximally. Continuing as such, we are eventually able to fold $E_1,\dots, E_n$. Since each fold was maximal, it either ended at a vertex or at a point that (now) has at least three gates. Thus, the number of vertices with less than three gates has decreased.

So it is possible via repeating the folding procedure to obtain a representative in which all vertices have at least three gates. Since the folds cannot increase the number of PNPs, the representative is PNP-free if $h$ was.  
\qedhere  
\end{proof}

\vskip10pt

%%%%%%%%%%%%%%%%%%%%%%%%%%%%%%%%%%%%%%%%%%%%%%%%%%%%%%%%%%%%%%%%%%%%
%%%%%%%%%%%%%%%%%%%%%%%%%%%%%%%%%%%%%%%%%%%%%%%%%%%%%%%%%%%%%%%%%%%%

\subsection{The directional surplus \& index deficit for fully preprincipal tt representatives}{\label{ss:MaxSing}}

%%%%%%%%%%%%%%%%%%%%%%%%%%%%%%%%%%%%%%%%%%%%%%%%%%%%%%%%%%%%%%%%%%%%

Recall from \S \ref{ss:indices} the index deficit $ID(\vphi)=i(\vphi)+r-1$ for a fully irreducible $\vphi\in\out$. The index deficit will relate to the ``directional surplus'' of its tt representatives:

Given a tt map $g\from\G\to\G$, we define the \emph{directional surplus} of $g$ by
\begin{equation}\label{eq:gatesurplus}
DS(g):=\sum_{v\in V \G} \left( \sum_{G~\in ~\text{Gates}(v)}(|G|-1) \right),
\end{equation}
where for each $v\in V\G$, we denote by $\text{Gates}(v)$ the set of gates at $v$.

Since a Stallings fold can only fold directions in the same gate, the directional surplus, roughly speaking, determines how many distinct fold choices one has. As in Example \ref{ex:StallingsFolds}, these distinct folds generally lead to a multitude of related Stallings fold decompositions. But, unless some gate has $\geq 3$ directions, the directional surplus bounds the number of choices of subsequent folds at each stage of forming a Stallings fold decomposition. The directional surplus gives a more complicated bound on possible folds in the case of gates with $\geq 3$ directions.\\

\begin{prop}[Index deficit for fully preprincipal tt representatives]\label{p:id}
Suppose that $r\geq 3$ and $g\from\G\to\G$ is a fully preprincipal tt representative of an ageometric fully irreducible $\vphi\in\out$. Then each of the following holds.
\begin{enumerate}[(a)]
\item $$ID(\vphi)~\geq ~\frac{1}{2}~|nPV \G|  ~+~ \frac{1}{2}~DS(g)~\geq ~\frac{1}{2}~|nPV \G|  ~+~ \frac{1}{2}.$$
\item If $g$ is fully singular, then $ID(\vphi)~= ~ \frac{1}{2}~DS(g)$. In particular, the directional surplus is the same for any two fully singular tt representatives of the same ageometric fully irreducible $\vphi\in\out$.
\end{enumerate}
\end{prop}

\begin{proof}
Suppose $g\from\G\to\G$ is a fully preprincipal tt representative of $\vphi$. Note that 
\begin{equation}\label{eq:vertices}
|V\G| ~=~ |PV \G|~+~|nPV \G|
\end{equation} 
and
\begin{equation}\label{eq:edges1}
|E\G|~ =~ \frac{1}{2}~|\mD\G|~ =~ \frac{1}{2}\sum_{v~\in ~ V \G}\left(\sum_{G~\in ~\text{Gates}(v)}|G|\right)
\end{equation}
$$~=~ \frac{1}{2}\sum_{v~\in ~ PV \G}|\text{Gates}(v)| ~+~ \frac{1}{2}\sum_{v~\in ~ nPV \G}|\text{Gates}(v)| ~+~ \frac{1}{2}~\text{DS}(g).$$

We first prove (a). Since $g$ is fully preprincipal, for each $v\in nPV \G$ we have $|\text{Gates}(v)|\geq 3$. So
\begin{equation}\label{eq:nonprincipalgates}
\sum_{v~\in ~ nPV \G}|\text{Gates}(v)| ~ \geq ~ 3~|nPV \G|.
\end{equation}

Combining (\ref{eq:edges1}) and (\ref{eq:nonprincipalgates}), we obtain
\begin{equation}\label{eq:edges2}
|E\G| ~\geq ~ \frac{1}{2}\sum_{v~\in ~ PV \G}|\text{Gates}(v)| ~+~ \frac{3}{2}~|nPV \G|  ~+~ \frac{1}{2}~\text{DS}(g).
\end{equation}

Since $\chi(G) ~=~ |V \G|~-~|E\G|$, we have

$$\chi(G) ~\leq ~ |PV \G|~+~|nPV \G|  ~-~ 
\left( \frac{1}{2}\sum_{v~\in ~ PV \G}|\text{Gates}(v)| ~+~ \frac{3}{2}~|nPV \G|  ~+~ \frac{1}{2}~\text{DS}(g) \right) $$
$$ ~=~ 
 |PV \G| ~-~ 
\frac{1}{2}\sum_{v~\in ~ PV \G}|\text{Gates}(v)| ~-~ \frac{1}{2}~|nPV \G|  ~-~ \frac{1}{2}~\text{DS}(g).$$

Since $\chi(G)~=~1-r$, this gives
\begin{equation}\label{eq:ec}
1-r  ~\leq ~ |PV \G| ~-~ 
\frac{1}{2}\sum_{v~\in ~ PV \G}|\text{Gates}(v)| ~-~ \frac{1}{2}~|nPV \G|  ~-~ \frac{1}{2}~\text{DS}(g). 
\end{equation}

Now, 
\begin{equation}\label{eq:indexum1}
i(\vphi)~ =~ \sum_{v~\in ~ PV \G} \left(1-\frac{|\text{Gates}(v)|}{2}\right) ~=~ |PV \G|~-~\frac{1}{2}\sum_{v~\in ~ PV \G}|\text{Gates}(v)|.
\end{equation}

And so
\begin{equation}\label{eq:final}
 ID(\vphi)~=~i(\vphi) ~-~ (1-r ) ~ \geq
\end{equation} 
$$|PV \G|~-~\frac{1}{2}\sum_{v~\in ~ PV \G}|\text{Gates}(v)| ~-~ 
 \left(|PV \G| ~-~ 
\frac{1}{2}\sum_{v~\in ~ PV \G}|\text{Gates}(v)| ~-~ \frac{1}{2}~|nPV \G|  ~-~ \frac{1}{2}~\text{DS}(g) \right).$$
And so $ID(\vphi)~\geq~\frac{1}{2}~|nPV \G|  ~+~ \frac{1}{2}~\text{DS}(g)$. 

Since $g$ is not a homeomorphism, some $g$-gate must have $\geq 2$ directions. Thus $DS(g)\geq 1$ and so, in fact
$$ID(\vphi)~\geq ~\frac{1}{2}~|nPV \G|  ~+~ \frac{1}{2},$$
proving (a).

We now prove (b). Now $g$ is fully singular and so
$$
\sum_{v~\in ~ nPV \G}|\text{Gates}(v)| ~ = ~ 0
$$
and thus
$$
|E\G| ~= ~ \frac{1}{2}\sum_{v~\in ~ V \G}|\text{Gates}(v)| ~+~ \frac{1}{2}~\text{DS}(g).
$$
So now
$$
\chi(G) ~= ~ |V \G|  ~-~ 
\frac{1}{2}\sum_{v~\in ~ V\G}|\text{Gates}(v)|  ~-~ \frac{1}{2}~\text{DS}(g)~~~~\text{   and   }~~~~
i(\vphi) ~=~ |V \G|~-~\frac{1}{2}\sum_{v~\in ~ V \G}|\text{Gates}(v)|.
$$
So
$$
 ID(\vphi)~=~i(\vphi) ~-~ \chi(G) ~ = ~ \left(|V \G|~-~\frac{1}{2}\sum_{v~\in ~ V \G}|\text{Gates}(v)|\right) ~-~ \left(|V \G|  ~-~ 
\frac{1}{2}\sum_{v~\in ~ V \G}|\text{Gates}(v)|  ~-~ \frac{1}{2}~\text{DS}(g) \right).
$$
And so $ID(\vphi)~=~ \frac{1}{2}~\text{DS}(g)$, as desired. \\
\qedhere  
\end{proof}

\vskip10pt

\begin{rmk} From the proof one can see that $ID(\vphi)$ is impacted by the number of nonprincipal vertices with $\geq 3$ gates and the direction surplus $\text{DS}(g)$, i.e. the number of directions in gates containing multiple directions. As $\text{DS}(g)$ also relates to the number of choices of folds in a Stallings fold decomposition, one can conjecture that $ID(\vphi)$ would give bounds on the dimension of the \cite{loneaxes} stable axis bundle.
\end{rmk}

\vskip20pt

%%%%%%%%%%%%%%%%%%%%%%%%%%%%%%%%%%%%%%%%%%%%%%%%%%%%%%%%%%%%%%%%%%%%
%%%%%%%%%%%%%%%%%%%%%%%%%%%%%%%%%%%%%%%%%%%%%%%%%%%%%%%%%%%%%%%%%%%%

\subsection{Pff decompositions of fully singular train track representatives}{\label{ss:MaxSingPffDecomp}}

%%%%%%%%%%%%%%%%%%%%%%%%%%%%%%%%%%%%%%%%%%%%%%%%%%%%%%%%%%%%%%%%%%%%

\begin{prop}[Pff decompositions of fully singular train track representatives]\label{prop:SingMax}
Suppose $g$ is a fully singular tt representative of an ageometric fully irreducible $\vphi\in\out$. Then $g$ has either a pff decomposition or a Stallings fold decomposition with at least one tripod proper full fold. 

If a fully preprincipal tt representative of an ageometric fully irreducible $\vphi\in\out$ has a pff decomposition then it is fully singular.
\end{prop}

\begin{proof} 
Suppose $g$ is a fully singular tt representative of an ageometric fully irreducible $\vphi\in\out$ and perform as many proper full folds as possible to decompose $g$ as $g=g''\circ g'$, where $g'\from\G\to\G'$ decomposes entirely into proper full folds and no Stallings fold decomposition of $g''\from\G'\to\G$ can start with a proper full fold. If $g''$ is a homeomorphism, then the proper full fold decomposition of $g'$ is the desired pff decomposition. Now suppose that $g'$ is not a homeomorphism.

Note that $\G'$ has the same number of vertices as $\G$, i.e. the minimal number of vertices possible for a train track representative of $\vphi$. Thus, simultaneously folding maximally each gate of $g''$ cannot lead to a complete fold of two edges. Since the decomposition of $g''$ cannot start with a proper full fold, this leaves that each of the folds is a partial fold. 

It is possible for the partial folds to pass over each other. However, unless two of the partial folds end at a common point (as in a tripod fold), no further folding of $g$ can occur. Thus, since the partial folds would have increased the number of vertices and $g'$ preserved the number of vertices, and $g''$ was not a homeomorphism, there exists an edge $e$ such that a partial fold of $\{e,e'\}$ and a partial fold of $\{\bar{e},e''\}$ terminate at a common point $p$ in $e$. Write $e=e_1 e_2$, and $e'=e_1' e_2'$, and $e''=e_1'' e_2''$ so that $g''(e_1) = g''(e_1')$ and $g''(\overline{e_2}) = g''(e_1'')$. If either $\{e_2',e_2''\}$ can only partially be folded or not folded at all, then we again have a contradiction with the increased number of vertices. If $e_2'$ and $e_2''$ are completely folded, then $g''$ would identify the terminal vertices of $e_2'$ and $e_2''$. Note that these vertices could not be equal or the fold would change the homotopy type of $\G'$. So completely folding $e_2'$ and $e_2''$ would contradict that $g$ and $g'$ are bijective on vertices, forcing $g''$ to be bijective on vertices. Thus, the fold is a tripod proper full fold.

We now prove the final sentence. No proper full fold can identify two vertices. Thus, any tt map $g$ with a pff decomposition is bijective on vertices and thus each vertex is periodic. Since each vertex of a preprincipal tt representative has $\geq 3 $ gates, we have that each vertex is in fact principal and $g$ is fully singular, as desired.
\qedhere
\end{proof}

\vskip25pt

%%%%%%%%%%%%%%%%%%%%%%%%%%%%%%%%%%%%%%%%%%%%%%%%%%%%%%

\section{Fully singular pff ltt structures \& maps}
	\label{s:PffLttStructures}

%%%%%%%%%%%%%%%%%%%%%%%%%%%%%%%%%%%%%%%%%%%%%%%%%%%%%%

%%%%%%%%%%%%%%%%%%%%%%%%%%%%%%%%%%%%%%%%%%%%%%%%%%%%%%

\subsection{Ltt structures of fully singular pff decompositions}
	\label{ss:PffLttStructures}

%%%%%%%%%%%%%%%%%%%%%%%%%%%%%%%%%%%%%%%%%%%%%%%%%%%%%%

\smallskip

The following two lemmas will help us to understand the ltt structures of fully singular tt representatives.\\

\begin{lem}\label{lemma:FSlttstructures} Suppose that $g\from \G\to\G$ is a fully singular tt representative of an ageometric fully irreducible $\vphi\in\out$ and that  $\G_{0} \xrightarrow{g_1} \G_{1} \xrightarrow{g_2} \cdots \xrightarrow{g_{n-1}} \G_{n-1} \xrightarrow{g_n} \G_n$ is a pff decomposition of $g$. Then $\mG(g)$ is a birecurrent abstract ltt structure $\mG$ with underlying graph $\G$ and index $i(\vphi)$.
\end{lem}

\begin{proof} We prove that each property holds, following the definition of \S \ref{s:AbstractLttStructures}, and noting that (ltt-i) and (ltt-iv) are just part of the definition, so require no proof.\\

(ltt-ii): Since the representative is fully singular, the number $R$ of nonperiodic directions is $DS(g)$. By Proposition \ref{p:id}, $ID(\vphi)=\frac{1}{2} DS(g)$, so $DS(g)=2 ID(\vphi)$.
By definition, $ID(\varphi) = i(\varphi) + r - 1$, so $DS(g)=2(i(\varphi) + r - 1)$, proving (ltt-ii). \\

(ltt-iii) This follows from Lemma \ref{lem:graphmap}: By (a$'$), $Dg_n(e_{j_{n-1}})\notin Image(Dg)$, so is represented by a red vertex. By (b$'$), the turn $\{Dg_n(\ol{e_{\ell_{n-1}}}),Dg_n(e_{j_{n-1}})\}$ is $g$-taken, so is represented by a red edge, containing the vertex $Dg_n(e_{j_{n-1}})$. By (c$'$), since $Dg_n(e_{j_{n-1}})\notin Image(Dg_{n,n-1})$ and $\{Dg_n(\ol{e_{\ell_{n-1}}}),Dg_n(e_{j_{n-1}})\}$ is the only turn of $\tau(g_{n,n-1})$, no other turn contains $Dg_n(e_{j_{n-1}})$.\\

(ltt-v) Most of (ltt-v) follows by the definition. Each vertex must be contained in a colored edge as follows. Suppose the vertex represents a direction $e$. Then $[e,e']$ is an edge in $\mG(g)$ if and only if some $g^k(e'')$ contains either $\ol{e'}e$ or $\ol{e}e'$. As $\vphi$ is fully irreducible, $g$ must be expanding with PF transition matrix so that, for an adequately high power $k$, the image of each edge maps over each other edge and contains multiple edges in its image. That is, for in fact each $e''\in E\G$, we have that $g^k(e'')$ passes over multiple edges including $e$. If $e$ is not the first edge in $g^k(e'')$, then we can take $e'$ to be the edge directly preceding $e$. Otherwise we can take $e'$ to be the edge directly following $e$.\\

(ltt-vi) This follows from the fact that local Whitehead graphs of fully irreducible outer automorphisms must be connected.\\

Finally, even though the definitions and situation are slightly different, the proof that $\mG(g)$ is birecurrent follows that of \cite[Proposition 4.4]{IWGI}, replacing the single local Whitehead graph in that situation with the union of the local Whitehead graphs of $\mG(g)$.

\qedhere

\end{proof}

\vskip25pt

\begin{lem}\label{lemma:MaximallySingularLTTstructures} Suppose $g\colon \G\to\G$ is a fully singular tt representative of an ageometric fully irreducible $\vphi\in\out$, with a pff decomposition, and $\mG(g)$ its associated ltt structure. Then
\begin{enumerate}[(a)]
\item the disjoint union of the purple graphs is the ideal Whitehead graph, and
\item the number of $g$-nonperiodic directions in $\G$ is $2(i(\vphi)+|E\G|-|V\G|)$, i.e. $2(i(\vphi)-\chi(G))$, and each of the nonperiodic directions occurs at a vertex of valence at least four, and
\item one of the nonperiodic directions is contained in precisely one turn taken by $g$, i.e one of the red directions is contained in precisely one red edge of  $\mG(g)$.
\end{enumerate}
\end{lem}

\begin{proof} We prove the statements one at a time.

a) Since $g$ is a fully singular tt representative of an ageometric $\vphi$, it has no PNPs, so $$\IW(\vphi)\cong\bigsqcup_{v\in V\Gamma} \SW(g;v).$$ 
Since the purple graphs are by definition the $\SW(g;v)$, (a) follows.

b) The first line follows from Lemma \ref{lemma:FSlttstructures}, so we prove the second line. Since each vertex is principal, it has $\geq 3$ periodic directions at it. Thus, a nonperiodic direction at a vertex would mean adding to the number of directions at the vertex, so there are $\geq 4$ directions.

c) This follows from Lemma \ref{lemma:FSlttstructures}.\\
\qedhere

\end{proof}

\vskip10pt

%%%%%%%%%%%%%%%%%%%%%%%%%%%%%%%%%%%%%%%%%%%%%%%%%%%%%%%%%%%%%%%%%%%%

\subsection{Maps of ltt structures}{\label{ss:maps_of_ltt_structures}}

%%%%%%%%%%%%%%%%%%%%%%%%%%%%%%%%%%%%%%%%%%%%%%%%%%%%%%%%%%%%%%%%%%%%

\subsubsection{Proper full fold ltt structure maps}{\label{ss:pffs_of_ltt_structures}}

Suppose $\mG$ is a pff ltt structure with underlying graph $\G$ and that $f\from\G\to\G'$ is a proper full fold in $\G$ of $e_1$ over $e_2$. Suppose further that either $e_1$ or $e_2$ is a red vertex direction of $\mG$. Then $f\cdot\mG$ is the pff ltt structure with 
\begin{itemize}
  \item[(pff-map-i.)] underlying graph $\G'$, and 
  \item[(pff-map-ii.)] $e_1$ a red direction, and 
  \item[(pff-map-iii.)] $[\ol{e_2}, e_1]$ a red edge, and 
  \item[(pff-map-iv.)] a colored edge $[Df(e_j),Df(e_k)]$ precisely when $[e_j,e_k]$ is a colored edge of $\mG$, and 
  \item[(pff-map-v.)] the further red directions are determined as follows:
  {\begin{itemize}
  \item[a.] if $e_1$ labels a red vertex in $\mG$, then vertex coloring is consistent betwixt $\mG$ and $f\cdot\mG$ and
  \item[b.] if $e_1$ labels a purple vertex in $\mG$, then  $e_1$ labels a red vertex in $f\cdot\mG$, and $e_2$ labels a purple vertex in $f\cdot\mG$, and all other vertex colors remain unchanged, and
\end{itemize}}
  \item[(pff-map-vi.)] the colored edges are red precisely if they contain a red vertex.
\end{itemize}

\vskip5pt

\begin{lem}[Pff ltt structure map]\label{lemma:MapsOfPffLttStructures} Suppose $\G_{0} \xrightarrow{g_1} \G_{1} \xrightarrow{g_2} \cdots \xrightarrow{g_{n-1}} \G_{n-1} \xrightarrow{g_n} \G_n$ is a pff decomposition of a fully singular tt representative of an ageometric fully irreducible $\vphi\in\out$. Let $g'\from\G_k\to\G_k$ denote $f_k$, i.e. $g'=g_{k,1}\circ g_{n, k+1}$. Then $\mG(g')=g_{k,1}\cdot\mG$ for each $k<n$.
\end{lem}

\begin{proof} 
It suffices to show the statement for $k=1$ and then the result follows by induction.

We use the standard notation that $g_1\from\G\to\G'$ is a proper full fold in $\G$ of $e_{j_1}$ over $e_{\ell_1}$, to be consistent with the notation of \S \ref{ss:maps_of_ltt_structures}. Thus, according to the standard notation, $g_1\from e_{j_1}\mapsto e_{\ell_1} e_{j_1}$ and $Dg_1\from e_{j_1}\mapsto e_{\ell_1}$. The translation to the definition's notation is $e_{j_1}=e_1$ and $e_{\ell_1}=e_2$.

(pff-map-i) follows from the definition of $\mG(g')$ and (pff-map-ii) follows from the proof of Lemma \ref{lem:graphmap}a.

(pff-map-iii) By (maps-ii), it suffices to show that $\{\ol{e_{\ell_1}}, e_{j_1}\}$ is $g'$-taken, which follows from the proof of Lemma \ref{lem:graphmap}c because  $\{\ol{e_{\ell_1}}, e_{j_1}\}\in\tau(g_1)$. 

(pff-map-iv) By the local Whitehead graph and ltt structure definitions, the colored edges of $\mG(g)$ and $\mG(g')$ correspond to, respectively, $\tau_{\infty}(g)$ and $\tau_{\infty}(g')$. Let $k\in\ZZ_{>0}$ be such that $\tau(g^k)=\tau_{\infty}(g)$ and $\tau(g'^k)=\tau_{\infty}(g')$. Then $g'^{k+1}=g_1\circ g^k\circ g_{n,2}$, and so, according to Lemma \ref{lem:graphmap}c, 
$$\tau(g'^{k+1})=\tau(g_1\circ g^k\circ g_{n,2})=\tau (g_n)\cup Dg_1(\tau(g^k))
\cup D(g_1\circ g^k)(\tau(g_{n,2})).$$
And $\tau (g_n)=[\bar{e_2}, e_1]$, so the remaining colored edges are $Dg_1(\tau(g^k))\cup D(g_1\circ g^k)(\tau(g_{n,2}))$.
And $Dg_1(\tau(g^k))=Dg_1(\tau_{\infty}(g))$ is precisely what we are hoping the rest of the taken turns are, so we are left to show that $D(g_1\circ g^k)(\tau(g_{n,2}))$ provides nothing new. Since direction maps of compositions are compositions of directions maps, $D(g_1\circ g^k)(\tau(g_{n,2}))=D(g_1(g^k)(\tau(g_{n,2})))$. Again, by Lemma \ref{lem:graphmap}c, $\tau(g_{n,2})\subseteq\tau(g^k)=\tau_{\infty}(g)$, so $D(g_1\circ g^k)(\tau(g_{n,2}))\subseteq Dg_1(\tau_{\infty}(g))$, as desired.

(pff-map-v) By Lemma \ref{lemma:MaximallySingularLTTstructures}b, $\mG(g)$ and $\mG(g')$ have the same number of red vertices. Further, $Image(Dg'^{R+1}) = Image(Dg_1\circ Dg^{R}\circ Dg_{n,2}) \subseteq Image(Dg_1\circ Dg^{R})$. 

Suppose first that $e_{j_1}$ is a red vertex direction of $\mG(g)$ and suppose that $e$ is another red vertex direction of $\mG(g)$. Then, $e_{j_1},e\notin Image(Dg^R)$ for any rotationless power $R$.   
 Since $e_{j_1}\notin Image(Dg^R)$ and $Dg_1$ is the identity on $\mD\G \backslash e_{j_1}$, we have that $Image(Dg_1\circ Dg^{R})=Image(Dg^{R})$. That is, $Image(Dg'^{R+1})\subseteq Image(Dg^{R})$. Since $\mG(g)$ and $\mG(g')$ have the same number of red vertices, this in fact means $Image(Dg'^{R+1}) = Image(Dg^{R})$. So $\mG(g)$ and $\mG(g')$ have the same red vertices.
 
Now suppose $e_{\ell_1}$ labels a red vertex of $\mG(g)$, but $e_{j_1}$ is not. That is, $Dg^R(e_{j_1})=e_{j_1}$ for any rotationless power $R$. Further, $e_{j_1}$ must be in the image of $Dg_{n,2}$ or could not be in the image of $Dg^R=D(g_{n,2}\circ g_1\circ g^{R-1})$. Keeping in mind that $Dg_1\from e_{j_1}\mapsto e_{\ell_1}$, this implies $e_{\ell_1}\in Image(D(g_1\circ g^{R}\circ g_{n,2})) = Image(Dg'^{R+1}).$ That is, $e_{\ell_1}$ now labels a purple vertex. In fact, with the added observation that $Dg_1$ is the identity on $\mD\G \backslash e_{j_1}$, we have by a very similar argument that all purple vertices in $\mG(g)$ stay purple, apart from $e_{j_1}$. Since $e_{j_1}\notin Image(D(g_1))$, we also have $e_{j_1}\notin Image(D(g_1)\circ g^{R}\circ g_{n,2}))= Image(Dg'^{R+1})$, i.e. $e_{j_1}$ is now a red direction vertex in $\mG(g')$.

(pff-map-vi) follows from the definition of $\mG(g')$.
\qedhere

\end{proof}

\vskip15pt

%%%%%%%%%%%%%%%%%%%%%%%%%%%%%%%%%%%%%%%%%%%%%%%%%%%%%%

\subsubsection{Edge-permutation ltt structure maps}
	\label{sss:EpLttStructureMaps}

%%%%%%%%%%%%%%%%%%%%%%%%%%%%%%%%%%%%%%%%%%%%%%%%%%%%%%

Suppose $\mG$ is a pff ltt structure with underlying graph $\G$. Suppose further that $\G'$ is graph-isomorphic to $\G$, and that $E\G$ and $E\G'$ are labeled using the same edge-labeling set $\{e_1,\dots, e_n\}$. Now suppose that $f_{\sigma}\from\G\to\G'$ is an edge-permutation graph isomorphism where $f_{\sigma}(e)=\sigma(e)$ for each $e\in E^{\pm}\G$. Since the vertices of $\mG$ are labeled by $E^{\pm}\G$, it makes sense to define $f\cdot\mG$ as the pff ltt structure obtained from $\mG$ by applying $\sigma$ to the vertex (and consistently black edge) labels.

\begin{lem}[Edge-permutation ltt structure map]\label{lemma:EdgePermutationLttStructureMaps} Suppose $\G_{0} \xrightarrow{g_1} \G_{1} \xrightarrow{g_2} \cdots \xrightarrow{g_{n-1}} \G_{n-1} \xrightarrow{g_n} \G_n$ is a pff decomposition of a fully singular tt representative of an ageometric fully irreducible $\vphi\in\out$. 
Let $g'\from\G_{n-1}\to\G_{n-1}$ denote $g'=g_{n-1,1}\circ g_{n}$. 
Then $g_n\cdot\mG(g')=\mG(g)$.
\end{lem}

\begin{proof} Suppose $g_n=f_{\sigma}$. It suffices to show that $g'$ takes a turn $\{d_1,d_2\}$ if and only if $g$ takes the turn $\{\sigma(d_1), \sigma(d_2)\}$ and a direction $d$ is $g'$-periodic if and only if $\sigma(d)$ is $g$-periodic. 

Let $p$ be such that $\tau(g'^p)=\tau_{\infty}(g')$ and $\tau(g^p)=\tau_{\infty}(g)$. 
We know that $g'^p$ is obtained from $g^p$ by
$g'^p\from\sigma^{-1}(e)\mapsto \sigma^{-1}(E_1)\dots\sigma^{-1}(E_m)$ if and only if
$g^p\from e\mapsto E_1\dots E_m$. 
Thus $g'$ takes a turn $\{d_1, d_2\}$ if and only if $g$ takes the turn $\{\sigma(d_1), \sigma(d_2)\}$.

Further, for a rotationless power $R$, we have that $g'^R\from e\mapsto e\dots$ if and only if $g^R\from \sigma(e)\mapsto \sigma(e)\dots$.
That is, $e$ is a periodic direction for $g'$ if and only if $\sigma(e)$ is a periodic direction for $g$.
\qedhere

\end{proof}

\vskip15pt

%%%%%%%%%%%%%%%%%%%%%%%%%%%%%%%%%%%%%%%%%%%%%%%%%%%%%%

\subsubsection{Special ltt structure maps: tt-friendly symmetry maps \& proper full folds}
	\label{sss:LttStructureMaps}

%%%%%%%%%%%%%%%%%%%%%%%%%%%%%%%%%%%%%%%%%%%%%%%%%%%%%%

Even abstract ltt structures have underlying graphs, so that we can talk about maps of these graphs, and the maps of ltt structures they induce. We will be particularly interested in two types of such maps that will compose to give tt maps.

Suppose that $\mG$ is a lone axis ltt structure with underlying graph $\G$ and the black edges of $\mG$ are labeled with $\{e_1,\cdots, e_n\}$, where $|E\G|=n$.

A \emph{tt-friendly symmetry map} of $\mG$ is a triple $(\mG,f,f\cdot\mG)$ such that $f$ is a graph isomorphism and $f\cdot\mG$ is color-preserving (possibly black edge orientation-reversing) graph isomorphic to  $\mG$.

A \emph{tt-friendly proper full fold (pff)} of $\mG$ is a pair $(\mG,f\cdot\mG)$ such that there exist distinct directions $E,E'\in\{e_1,\ol{e_1},\cdots, e_n,\ol{e_n}\}$ satisfying that
      {\begin{enumerate}
        \item[(tt-pff-i)] $E$ and $E'$ label vertices in the same component of the colored subgraph of $\mG$, and
        \item[(tt-pff-ii)] either $E$ or $E'$ labels a red vertex, and
        \item[(tt-pff-iii)] $[E', E]$ is not a colored edge of $\mG$ and $\ol{E'}\neq E$, and
        \item[(tt-pff-iv)] $f$ is a proper full fold of $E$ over $E'$ in $\G$.
      \end{enumerate}}
      
In the case of lone axis ltt structures, which will be describe in \S \ref{ss:loneaxesautomata}, (ii) and (iii) can be rewritten as $v_r\in\{E,E'\}$ and $e_r\neq[\ol{E'}, E]$, with the added condition that $\ol{E'}\neq E$.
      
\vskip15pt

%%%%%%%%%%%%%%%%%%%%%%%%%%%%%%%%%%%%%%%%%%%%%%%%%%%%%%

\subsubsection{Compositions of ltt structure maps}
	\label{sss:CompositionsOfLttStructureMaps}

%%%%%%%%%%%%%%%%%%%%%%%%%%%%%%%%%%%%%%%%%%%%%%%%%%%%%%

The following two lemmas will help us to translate between pff decompositions of fully singular tt representatives and loops in the ltt structure automata.

\begin{lem}{\label{l:PffdecompsTTfriendly}} Suppose $g$ is a fully singular tt representative of an ageometric fully irreducible $\vphi\in\out$. And suppose that 
$\G_0 \xrightarrow{g_{1}}\G_1 \xrightarrow{g_{2}} \dots \xrightarrow{g_{n}}\G_n=\G_0$ is a pff decomposition of $g$. Let $\mG(f_k)=\mG_k$ for each $k$. Then
$\mG_{k-1} \xrightarrow{g_k}\mG_k$ is a tt-friendly proper full fold for each $k<n$ and either a tt-friendly proper full fold or tt-friendly symmetric map for $k=n$.
\end{lem}

\begin{proof} 
By Lemma \ref{lemma:MapsOfPffLttStructures}, $\mG_k=g_{k,1}\cdot\mG=\mG(f_k)$.
Thus, by Lemma \ref{lemma:FSlttstructures} and recalling that the rotationless index is an outer automorphism conjugacy class invariant, each $\mG_k$ is a birecurrent abstract ltt structure with underlying graph $\G_k$ and index $i(\vphi)$. In fact, since the ideal Whitehead graph is also a conjugacy class invariant, $IW(\mG_k)=IW(\mG)$ for each $k$.

Suppose $g_k$ is a proper full fold of $E$ over $E'$ for some edges $E,E'\in E\G_{k-1}$. Then (tt-pff-i) is satisfied because $E$ and $E'$ must emanate from the same vertex of $\G_{k-1}$ and the colored subgraph is the disjoint union of the $LW(f_{k-1})$. If both $E$ and $E'$ labeled purple vertices, then both $E$ and $E'$ would be periodic directions for $f_{k-1}$, meaning that they would be fixed by some $(f_{k-1})^m$. But, because $E$ and $E'$ are folded, they have the same image under $Dg_k$, hence under $D(f_{k-1})^m$, contradicting that $E\neq E'$. Thus (tt-pff-ii) holds. Finally, (tt-pff-iii) holds because $f_{k-1}$ is a tt map and (tt-pff-iv) holds by our assumption that $g_k$ is a proper full fold of $E$ over $E'$. 

It is possible that $g_n$ is a tight homeomorphism that is not the identity map. A tight homeomorphism will be a graph isomorphism, so we are left to show $g_n$ induces a color-preserving (possibly black edge orientation-reversing) graph isomorphism $\mG_{n-1}\to\mG_{n}$. This follows from Lemma \ref{lemma:EdgePermutationLttStructureMaps}.\\
\qedhere

\end{proof}

\vskip10pt

While each loop in the automata will describe a tt map and has the potential to define a fully singular tt representative of a fully irreducible $\vphi\in\out$, the loop may not compose enough of the right kinds of generators to have the entire colored graph realized as turns taken or to identify enough directions, for example. Thus, the following lemma is needed to establish the necessary criteria for a loop to define a fully singular tt representative of a fully irreducible $\vphi\in\out$ with the desired ideal Whitehead graph.

\vskip10pt

\begin{lem}\label{lemma:LoopsAreTtMapsLA} Suppose $\mG_0 \xrightarrow{g_{1}}\mG_1 \xrightarrow{g_{2}} \dots \xrightarrow{g_{n}}\mG_n=\mG_0$, where each $\mG_k \xrightarrow{g_{k+1}}\mG_{k+1}$ is either a tt-friendly symmetry map or tt-friendly pff. Suppose that $\G_k$ is the underlying graph of $\mG_k$ for each $k$, with $\G_0=\G_n$ also denoted by $\G$. And denote also $\mG_0=\mG_n$ by $\mG$.
Then 
\begin{itemize}
  \item[(a.)] $g_n\circ\cdots\circ g_1\from\G\to\G$ is a tt map and
  \item[(b.)] each turn of $\tau_{\infty}(g)$ is represented by a colored edge in $\mG$.
\end{itemize} 
If $g$ additionally
\begin{enumerate}
%\item takes each turn of $\mG$, and
  \item has no PNPs, and
  \item has that each $LW(g,v)$ is graph isomorphic to the appropriate components of the colored graph of the ltt structure, and
  \item has a PF transition matrix, and
  \item has $\geq 3$ periodic directions at each vertex 
\end{enumerate}
then, appropriately marked, $g$ is a fully singular tt representative of an ageometric fully irreducible $\vphi\in\out$.

If further yet, 
\begin{enumerate}\addtocounter{enumi}{4}
\item $Image(Dg)$ is precisely the set of purple directions of $\mG$,
\end{enumerate}
\noindent then:

{\begin{itemize}
  \item[(i.)]  $IW(\vphi)=\sqcup SW(\mG,v)$ and
  \item[(ii.)] the red vertex directions are precisely those not in $Image(Dg^R)$ for a rotationless power $R$.
\end{itemize} }
\end{lem}

\begin{proof} For each $k\in\ZZ_{k\geq 0}$, let $g_k=g_{k~(\text{mod } n)}$ so that, for example, $g_{1,mn}=g^m$.
We proceed by proving inductively the following holds for each $k$:\\
\indent (1) $g_{k,1}(e)$ is tight for each $e\in E\G$ and\\
\indent (2) each turn of $\tau(g_{1,k})$ is represented by a colored edge in $\mG_k$.

Suppose the inductive hypothesis holds and that $e\in E\G$ is arbitrary. Suppose that $g_{k,1}(e) = E_1\dots E_N$. Then $g_{k+1,1}(e) = g_{k+1}(E_1\dots E_N) = g_{k+1}(E_1)\dots g_{k+1}(E_N)$. Thus, to show that $g_{k+1,1}(e)$ is tight, we need that $g_{k+1}$ is tight (which it is because graph isomorphisms and proper full folds are) and that each $\{\ol{g_{k+1}(E_{\ell})},g_{k+1}(E_{\ell+1})\}$ is nondegenerate. We prove now the latter. Each $\{\ol{E_{\ell}}, E_{\ell+1}\}$ is a $g_{k,1}$-taken turn and so, by (2) in the inductive hypothesis, is represented by a colored edge in $\mG_k$. As in the notation of the tt-friendly pff definition, we suppose $g_{k+1}$ maps $E$ over $E'$, i.e. $E\mapsto E'E$. In other words, $Dg_{k+1}\from E\mapsto E'$ and fixes each other direction of $\G_k$. Thus, in order for $\{\ol{E_{\ell}}, E_{\ell+1}\}$ to be degenerate, one would need that $\{\ol{E_{\ell}}, E_{\ell+1}\}=\{E, E'\}$, where we are still considering turns to be unordered. Since $\{\ol{E_{\ell}}, E_{\ell+1}\}$ is a colored edge of $\mG$, this would imply that $\{E, E'\}$ is a colored edge of $\mG$. However, this would contradict (tt-pff-iii). So (1) of the inductive hypothesis holds for $k+1$.

We are left to prove that each turn of $\tau(g_{1,k+1})$ is represented by a colored edge in $\mG_k$. Suppose that $\{d_1,d_2\}\in\tau(g_{1,k+1})$. Then there exists an $e\in E\G$ and $\ell$ such that $\{d_1,d_2\}=\ol{E_{\ell}}, E_{\ell+1}$, where $g_{k+1,1}(e)=E_1\dots E_N$. In the case that $g_{k+1}$ is a tt-friendly symmetry map, then there are $E_1',\dots, E_N'\in E\G$ such that $E_1\dots E_N=g_{k+1}(E_1')\dots g_{k+1}(E_N')$. So in the case that $g_{k+1}$ is a tt-friendly symmetry map, the result follows from (maps-iv). We now assume that $g_{k+1}$ is a tt-friendly proper full fold and again assume the notation of the definition, i.e. $g_{k+1}$ maps $E$ over $E'$, so that $g_{k+1}\colon E\mapsto E'E$ and $Dg_{k+1}\from E\mapsto E'$.
Now, $Dg_{k+1}(E)$ takes the turn $\{\ol{E}, E'\}$ and $Dg_{k+1}(e)=e$ for each $e\in\G_k\backslash \{E,\ol{E'}\}$. This gives us that, if $g_{k+1}(e)=g_k(E_1)\dots g_k(E_N)$, then the only turns taken by $g_{k+1}(e)$ are the $\{\ol{Dg_k(E_j)},Dg_k(E_{j+1})\}$ and possibly $\{\ol{E}, E'\}$. But (maps-iv) implies that each $\{\ol{Dg_k(E_j)},Dg_k(E_{j+1})\}$ is a colored edge in $f\cdot\mG$ and $\{\ol{E}, E'\}$ is red by (maps-iii).

If $g$ additionally satisfies (1)-(4), then Proposition \ref{prop:FIC} implies $g$ represents an ageometric fully irreducible $\vphi\in\out$.

Since $g$ has $\geq 3$ periodic directions at each vertex and has no PNPs, each vertex of $\G$ is principal and $g$ is fully singular.

We now assume that (5) additionally holds and prove (i-ii). By (3), we know that the components of the colored graph of $\mG$ correspond to the local Whitehead graphs of $g$ so that we are left to show that the $g$-periodic directions are precisely the purple vertices of $\mG$. By (6), $Image(Dg)$ is precisely the set of purple directions of $\mG$. Now, by definition, tt-friendly proper full folds cannot identify two purple directions (and tt-friendly symmetry maps are a bijection on directions), so $|Image(Dg)|=|Image(Dg^k)|$ for each $k$. Further, since $D(g^k)=Dg(Dg^{k-1})$, we have that $Image(Dg^k)\subseteq Image(Dg)$. So $Image(Dg^k)= Image(Dg)$ for each $k$ and, since red and purple directions are complementary sets, the proof is complete.\\
\qedhere

\end{proof}

\vskip25pt

%%%%%%%%%%%%%%%%%%%%%%%%%%%%%%%%%%%%%%%%%%%%%%%%%%%%%%%%%%%%%%%%%%%%
%%%%%%%%%%%%%%%%%%%%%%%%%%%%%%%%%%%%%%%%%%%%%%%%%%%%%%%%%%%%%%%%%%%%

%%%%%%%%%%%%%%%%%%%%%%%%%%%%%%%%%%%%%%%%%%%%%%%%%%%%%%

\section{False singularities \& Pff decomposition among multiple Stallings fold decompositions}
	\label{s:FalseSingularities}

%%%%%%%%%%%%%%%%%%%%%%%%%%%%%%%%%%%%%%%%%%%%%%%%%%%%%%

%%%%%%%%%%%%%%%%%%%%%%%%%%%%%%%%%%%%%%%%%%%%%%%%%%%%%%

\subsection{Merging outer automorphisms \& false singularities}
	\label{ss:FalseSingularities}

%%%%%%%%%%%%%%%%%%%%%%%%%%%%%%%%%%%%%%%%%%%%%%%%%%%%%%

Recall from Proposition \ref{prop:SingMax} that a fully singular tt representative or an ageometric fully irreducible either has a pff decomposition or a representative with a tripod fold.
Since a tripod fold increases the number of vertices, a Stallings fold decomposition of a fully singular tt representative that contains a tripod fold will necessarily also contain a complete fold of 2 edges (identifying 2 vertices). Further, all turns at the vertex $v$ created by a tripod fold are taken, ensuring they are not folded by further folds in the decomposition. However, $v$ is not in the image of a vertex. We thus call these vertices created by tripod folds in Stallings fold decompositions of fully singular tt representatives \emph{false singularities}. We more generally call a vertex created in a Stallings fold decomposition, not in the image of a vertex, \emph{vanishing}. We do not discuss such vertices and decompositions beyond this section, but find them intriguing enough to compel our naming them upon their discovery.

We call an ageometric fully irreducible outer automorphism in which no fully singular tt representative has a pff decomposition \emph{merging}. We conjecture that merging outer automorphisms are rare, and particularly that most fully singular tt representative have a pff decomposition. With this in mind, we focus here on pff decompositions of fully singular tt representatives.

\vskip10pt

%%%%%%%%%%%%%%%%%%%%%%%%%%%%%%%%%%%%%%%%%%%%%%%%%%%%%%

\subsection{Pff decomposition among multiple Stallings fold decompositions}{\label{s:ExPurity}}

%%%%%%%%%%%%%%%%%%%%%%%%%%%%%%%%%%%%%%%%%%%%%%%%%%%%%%

We conclude this section with an example of how a tt representative of an ageometric fully irreducible outer automorphism can have multiple Stallings fold decompositions, some of which have vanishing vertices and one of which does not. In fact, one Stallings fold decomposition is a pff decomposition.

\begin{ex}{\label{ex:StallingsFolds}}
Consider the map $S$ on the 3-petaled rose defined by

\begin{equation}\label{eq:S}
S \colon
\begin{cases}
a\mapsto cbca \\
b \mapsto cbc\\
c \mapsto ac\\
\end{cases}
\end{equation}
The map $S$ is a tt map with 2 illegal turns and multiple Stallings fold decompositions, only one of which is a pff decomposition. The images of edges are indicated on each graph to illuminate which edge segments have the same image (so can be folded). The map $S$ runs from the upper left-hand corner to the lower right-hand corner.

\begin{center}
\noindent\includegraphics[width=6.5in]{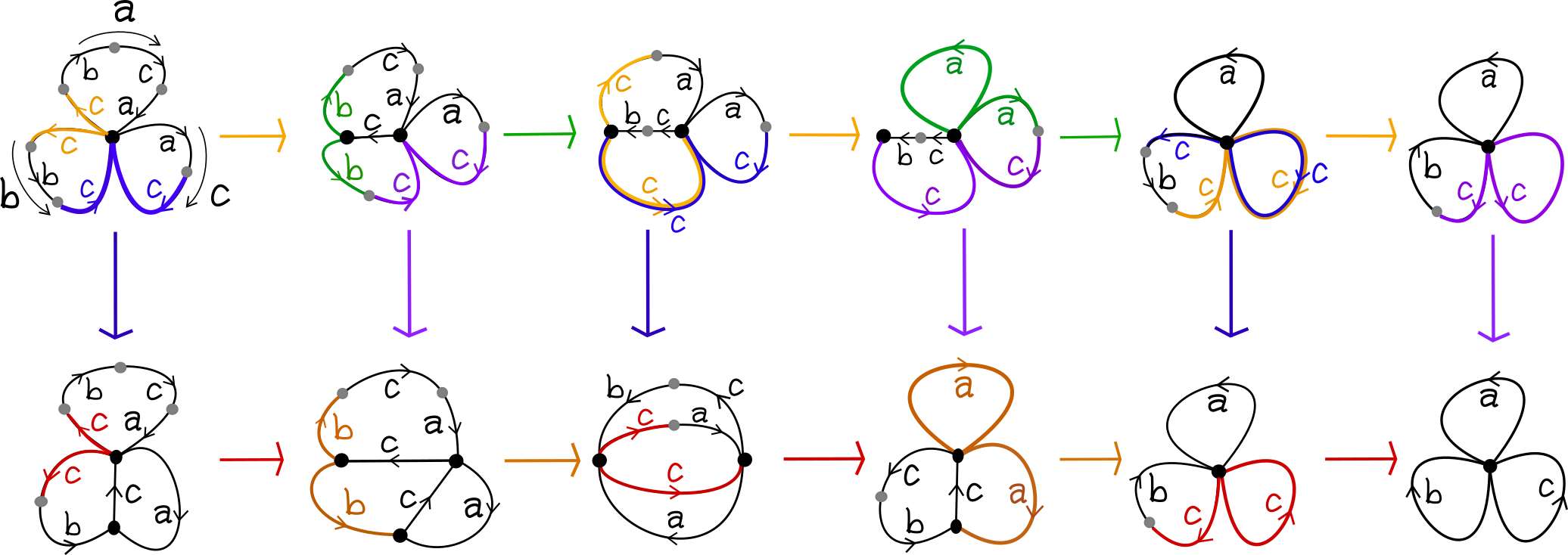}
\end{center}

The first 2 lower folds can combine into a single fold, and then going down the left side and then along the bottom row yields a Stallings fold decomposition where the number of vertices increases and then decreases, indicating the existence of vanishing vertices. 

The upper 3 folds combine to a single proper full fold, starting the pff decomposition of $S$ that runs along the top row and then down the right-hand side:

\begin{center}
\noindent\includegraphics[width=6.5in]{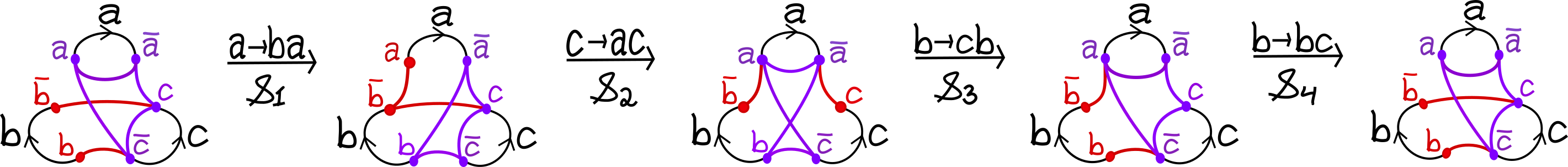}
\end{center}

This pff decomposition indicates the creation of vanishing vertices is unnecessary in this particular example. Unlike in the circumstance of a tripod proper full fold in Proposition \ref{prop:SingMax}, these vanishing vertices are a result of a partial fold chosen over a proper full fold.

\end{ex}

\smallskip

\begin{prop}\label{prop:multipledecomps} 
There exists a PNP-free tt representative of an ageometric fully irreducible $\vphi\in\outt$ that has both a pff decomposition and a Stallings fold decomposition including a complete fold of an edge pair. 
\end{prop}

\begin{proof}
We prove that $S$ represents an ageometric fully irreducible $\vphi\in\outt$ by proving the conditions of the Proposition \ref{prop:FIC} are satisfied. 

We know $g$ is a tt map because the images of positively oriented edges only contain positively oriented edges. The direction map is defined by \\

\noindent \hspace*{40mm} $DS\from a\mapsto c\mapsto a$ \quad\quad $DS\from \overline{a}\mapsto \overline{a}$\\
\hspace*{40mm} $DS\from b\mapsto c\mapsto a$ \quad\quad $DS\from \overline{b}\mapsto \overline{c}$\\
\hspace*{40mm} $DS\from c\mapsto a\mapsto c$ \quad\quad $DS\from \overline{c}\mapsto \overline{c}$\\

\noindent Thus, the illegal turns are $\{a,b\}$ and $\{\overline{b},\overline{c}\}$ and all directions but $b$ and $\overline{b}$ are periodic. Further, $\tau(S)=\{\{\overline{c},b\}, \{\overline{b},c\}, \{\overline{c},a\}, \{\overline{a},c\}\}$. And applying $DS$, we obtain $\tau(S^2)=\{\{\overline{c},c\}, \{\overline{c},a\}, \{\overline{a},a\}\}$ and applying $DS^2$ we obtain $\tau(S^2)=\{\{\overline{c},a\}, \{\overline{c},c\}, \{\overline{a},c\}\}$. Since $S^2$ fixes all periodic directions, 
$$\tau_{\infty}(S)=\{\{\overline{c},b\}, \{\overline{b},c\}, \{\overline{c},a\}, \{\overline{c},c\}, \{\overline{a},c\}, \{\overline{a},a\}\}.$$

Thus, the only local Whitehead graph is connected, the red edges of $\mG(S)$ are $[\overline{c},b]$ and $[\overline{b},c]$, and the purple edges of $\mG(S)$ are $[\overline{c},a]$, $[\overline{c},c]$, $[\overline{a},c]$, and $[\overline{a},a]$. 

For $S^2$ each edge is in the image of each edge so that the transition matrix $M(S)$ is PF and we are left to show that $S$ has no PNPs, or rather that no power has an iNP.

If some $S^k$ had an iNP $\rho$, there would exist legal paths $\rho_1$ and $\rho_2$ so that $\rho=\overline{\rho_1}\rho_2$ and $\{D\rho_1,D\rho_2\}$ is an illegal turn.
In particular, we repeatedly use Lemma \ref{l:pnpelimination} and Lemma \ref{l:longinps} to prove such a $\rho$ cannot exist.

As established in \S \ref{ss:StallingsFoldDecompositions}, for each $k$ we let $f_k=\mathfrak{s}_k\circ\cdots\circ \mathfrak{s}_1\circ \mathfrak{s}_n\circ\cdots\circ \mathfrak{s}_{k+1}$, i.e. $f_k=\mathfrak{s}_{k,1}\circ \mathfrak{s}_{n,k+1}$. We will use the below chart summarizing relevant data.

\begin{table}[h]
\begin{tabular}{|l|l|l|l|l|}
\hline 
 &  &  &  &   \\
 & \quad $f_0=S$ & \quad \quad $f_1$ & \quad \quad $f_2$ & \quad \quad $f_3$   \\
 
 &  &  &  &     \\
 \hline 
 &  &  &  &     \\
\text{illegal turns} & $\{a,b\}$, $\{\ol{b},\ol{c}\}$ & $\{a,c\}$, $\{\ol{b},\ol{c}\}$ & $\{b,c\}$, $\{\ol{b},\ol{c}\}$ & $\{a,b\}$, $\{\ol{b},\ol{c}\}$ \\
 &  &  &  &    \\
\hline 
 &  &  &  &    \\
\text{Unachieved directions} & \quad $b$~~\& ~~ $\ol{b}$ & \quad $a$~~\& ~~ $\ol{b}$  & \quad $c$~~\& ~~ $\ol{b}$ & \quad $b$~~\& ~~ $\ol{b}$ \\
 &  &  &  &    \\
\hline 
\hline 
 &  &  &  &   \\
 & \quad $S_{1,1}=\mathfrak{s}_1$ & \quad \quad $S_{2,1}$ & \quad \quad $S_{3,1}$ & \quad $S_{4,1}=S$   \\
 
 &  &  &  &     \\
 \hline 
 &  &  &  &     \\
 & \quad $a\mapsto ba$ & \quad $a\mapsto ba$ & \quad $a\mapsto cba$ & \quad $a\mapsto cbca$ \\
 & \quad $b\mapsto b$  & \quad $b\mapsto b$  & \quad $b\mapsto cb$  & \quad $b\mapsto cbc$    \\
 & \quad $c\mapsto c$  & \quad $c\mapsto ac$  & \quad $c\mapsto ac$  & \quad $c\mapsto ac$    \\
 &  &  &  &     \\
\hline 
\end{tabular}
\end{table}

We begin by assuming $\{D\rho_1,D\rho_2\}=\{a,b\}$, so that there exist edges $e_i,e_j'\in E\G$ so that $\rho_1=a e_2\dots e_n$ and $\rho_2=b e_2'\dots e_m'$. Note that each turn of $\rho_1$ and $\rho_2$ must be $S$-taken. Now

$\mathfrak{s}_1(\rho_1) = \mathfrak{s}_1(a e_2\dots e_n) = ba \mathfrak{s}_1(e_2)\dots \mathfrak{s}_1(e_n)$ and

$\mathfrak{s}_1(\rho_2) = \mathfrak{s}_1(b e_2'\dots e_m') = b \mathfrak{s}_1(e_2')\dots \mathfrak{s}_1(e_m')$.\\
So we need that $\{a, \mathfrak{s}_1(e_2')\}$ is either degenerate or an illegal turn for $f_1$, i.e. $\mathfrak{s}_1(e_2')=a$ or $\mathfrak{s}_1(e_2')=c$. Since $a\notin Image(D\mathfrak{s}_1)$ and only $D\mathfrak{s}_1(c)=c$, we have $e_2'=c$, i.e. $\rho_2 = bc e_3'\dots e_m'$. So

$S_{2,1}(\rho_1) = S_{2,1}(a e_2\dots e_n) = ba S_{2,1}(e_2)\dots S_{2,1}(e_n)$ and

$S_{2,1}(\rho_2) = S_{2,1}(bc e_3'\dots e_m') = bac S_{2,1}(e_3')\dots S_{2,1}(e_m')$.\\
So we need that $\{S_{2,1}(e_2), c \}$ is either degenerate or an illegal turn for $f_2$, i.e. $S_{2,1}(e_2)=b$ or $S_{2,1}(e_2)=c$. 
Since $c\notin Image(DS_{2,1})$, and only $DS_{2,1}(a), DS_{2,1}(b) =b$, and $\{\ol{a},b\}\notin\tau_{\infty}(S)$, we have  $e_2=a$, i.e. $\rho_1=aa e_3\dots e_n$. So

$S_{3,1}(\rho_1) = S_{3,1}(aa e_3\dots e_n) = cbacba S_{3,1}(e_3)\dots S_{3,1}(e_n)$ and

$S_{3,1}(\rho_2) = S_{3,1}(bc e_3'\dots e_m') = cbac S_{3,1}(e_3')\dots S_{3,1}(e_m')$.\\
So we need $\{b, S_{3,1}(e_3') \}$ is either degenerate or an illegal turn for $f_3$, i.e. $S_{3,1}(e_3')=b$ or $S_{3,1}(e_3')=a$. 
Since $b\notin Image(DS_{3,1})$ and only $DS_{3,1}(c) =a$, we have $e_3'=c$, i.e. $\rho_2=bcc e_4'\dots e_m'$. So

$S_{4,1}(\rho_1) = S_{4,1}(aa e_3\dots e_n) = cbcacbca S_{4,1}(e_3)\dots S_{4,1}(e_n)$ and

$S_{4,1}(\rho_2) = S_{4,1}(bcc e_4'\dots e_m') = cbcacac S_{4,1}(e_4')\dots S_{4,1}(e_m')$.\\
Cancellation ends with $\{a, b \}$, which is an illegal turn for $S$, so we apply $\mathfrak{s}_5 = \mathfrak{s}_1$ to reach

$S_{5,1}(\rho_1) = cbcbacbcba S_{5,1}(e_3)\dots S_{5,1}(e_n)$ and

$S_{5,1}(\rho_2) = cbcbacbac S_{5,1}(e_4')\dots S_{5,1}(e_m')$.\\
Now cancellation ends at $\{a, c \}$, which is an illegal turn for $f_1$, so we apply $\mathfrak{s}_6 = \mathfrak{s}_2$:

$S_{6,1}(\rho_1) = acbacbaacbacba S_{6,1}(e_3)\dots S_{6,1}(e_n)$ and

$S_{6,1}(\rho_2) = acbacbaacbaac S_{6,1}(e_4')\dots S_{6,1}(e_m')$.\\
Cancellation has again ended at $\{a, c \}$, but $\{a, c \}$ is not an illegal turn for $f_2$. So we have reached a contradiction with Lemma \ref{l:pnpelimination}.  

So we instead assume $\{D\rho_1,D\rho_2\}=\{\ol{b},\ol{c}\}$ and now

$S(\rho_1) = S(\ol{b} e_2\dots e_n) = \ol{c} \ol{b}\ol{c} S(e_2)\dots S(e_n)$ and

$S(\rho_2) = S_{2,1}(\ol{c} e_2'\dots e_m') = \ol{c}\ol{a} S(e_2')\dots S(e_m')$.\\
But then cancellation ends with $\{\ol{b},\ol{a}\}$, which is not an illegal turn for $S$, and so the PNP $\rho$ could not have existed. 
\qedhere

\end{proof}

\vskip25pt

%%%%%%%%%%%%%%%%%%%%%%%%%%%%%%%%%%%%%%%%%%%%%%%%%%%%%%%%%%%%%%%%%%%%
%%%%%%%%%%%%%%%%%%%%%%%%%%%%%%%%%%%%%%%%%%%%%%%%%%%%%%%%%%%%%%%%%%%%

\section{Lone axis train track automata}{\label{ss:loneaxesautomata}}

%%%%%%%%%%%%%%%%%%%%%%%%%%%%%%%%%%%%%%%%%%%%%%%%%%%%%%%%%%%%%%%%%%%%

Examples of lone axis train track automata are introduced in \cite{IWGI}, \cite{automaton}, and \cite{wiggd1}. We provide a general description of them here before generalizing further to train track automata encoding more general pff decompositions. A unique aspect of the lone axis situation is that one has that all Stallings fold decompositions of all tt representatives of all lone axis fully irreducible outer automorphisms are fold-conjugate to tt representative Stallings fold decompositions realized as loops in the automata. For this we need the following lemma.

\begin{lem}[Lone Axis LTT Structures]\label{lemma:LoneAxisLttStructures} Suppose $\vphi\in\out$ is a lone axis fully irreducible outer automorphism, then each tt representative of $\vphi$ is partial-fold conjugate to a tt map $g\from \G\to\G$ satisfying: 
\begin{enumerate}[(a)]
\item the disjoint union of the purple graphs in $\mG(g)$ is the ideal Whitehead graph, and
\item all but one direction in $\G$ is $g$-periodic and this direction is at a vertex of $\G$ with valence $>3$, and
\item the nonperiodic direction is contained in precisely 1 turn taken by $g$.
\end{enumerate}
\end{lem}

\begin{proof}
Following the first three paragraphs of the proof of \cite[Lemma 3.2]{wiggd1}, using \cite[Corollary 3.8]{loneaxes} and via a partial fold, we obtain a tt representative $g$ of a rotationless power $\vphi^R$ for which all vertices and periodic directions are fixed and have $\geq 3$ fixed directions. Again, since $\vphi$ and $\vphi^R$ are lone axis fully irreducible outer automorphisms, this fold is within the shared axis $\mA$ of all tt representative of $\vphi^R$. Thus, all tt representatives of $\vphi^R$ are fold-conjugate to $g$.

Since $\vphi\in\out$ is a lone axis fully irreducible, its ideal Whitehead graph $\IW(\vphi)$, hence also $\IW(\vphi^k)$ for each $k\in\ZZ_{>0}$, has no cut vertices. Thus, by \cite[Lemma 4.5]{loneaxes}, no tt representative $\tau$ of any $\vphi^k$ has a PNP. Thus, $IW(\vphi^k)$ is the disjoint union of the $SW(\tau,v)$ having $\geq 3$ vertices. Since each vertex of $g^R$ has $\geq 3$ fixed directions, each vertex of $g$ has $\geq 3$ periodic directions. And so the disjoint union of the purple graphs is the ideal Whitehead graph, proving (a).

We now prove (b). By \cite[Lemma 3.6]{loneaxes}, $g$, and each $g^k$, has precisely one illegal turn and this illegal turn contains the unique nonperiodic direction. Since each vertex is principal, each vertex has $\geq 3$ periodic directions. Since one direction in the illegal turn is the nonperiodic direction, the vertex with the illegal turn must then have $>3$ vertices.

(c) follows from Lemma \ref{lem:graphmap}.
\qedhere

\end{proof}

\vskip15pt

In light of Lemma \ref{lemma:LoneAxisLttStructures}, we call an ltt structure satisfying all of (a)-(c) of Lemma \ref{lemma:LoneAxisLttStructures} a \emph{lone axis ltt structure}. And by a \emph{lone axis ltt structure} $\mG$ we will mean an abstract ltt structure for which:
\begin{itemize}
  \item[(ltt-vii)] there is precisely one red vertex, which we denote $v_r$, and
  \item[(ltt-viii)] the index satisfies $\mI (\mG)=\frac{3}{2}-r$, and
  \item[(ltt-ix)] no component of $IW(\mG)$ has a cut vertex.
\end{itemize}
Since $v_r$ is the only red vertex, it must be contained in precisely one (necessarily red) colored edge, which we denote $e_r$.

\vskip15pt

%%%%%%%%%%%%%%%%%%%%%%%%%%%%%%%%%%%%%%%%%%%%%%%%%%%%%%%%%%%%%%%%%%%%

\subsubsection{Folds induce maps of lone axis ltt structures}{\label{ss:maps_of_ltt_structures}}

Suppose $\mG$ is a lone axis ltt structure with underlying graph $\G$ and that $f$ is a proper full fold in $\G$ of $e_1$ over $e_2$. Suppose further that either $v_r=e_1$ or $v_r=e_2$ in $\mG$. Then $f\cdot\mG$ is the lone axis ltt structure with underlying graph $\G$, and $v_r=e_1$, and $e_r=[\bar{e_2}, e_1]$, and there is a purple edge $[Df(e_j),Df(e_k)]$ precisely when $[e_j,e_k]$ is a colored edge of $\mG$.

\begin{lem}[Images of ltt structures]\label{l:LAlttStructureMaps} Suppose $\vphi\in\out$ is a lone axis fully irreducible outer automorphism and $\G_{0} \xrightarrow{g_1} \G_{1} \xrightarrow{g_2} \cdots \xrightarrow{g_{n-1}} \G_{n-1} \xrightarrow{g_n} \G_n$ is a pff decomposition of a tt representative $g\from \G\to\G$ of $\vphi$. Let $g'\from\G_1\to\G_1$ denote $g_{k,1}\circ g_{n,k+1}$. Then $\mG(g')=g_{k,1}\cdot\mG$.
\end{lem}

\begin{proof} This is a special case of Lemma \ref{lemma:MapsOfPffLttStructures}.
\qedhere

\end{proof}

\vskip10pt

%%%%%%%%%%%%%%%%%%%%%%%%%%%%%%%%%%%%%%%%%%%%%%%%%%%%%%%%%%%%%%%%%%%%

\subsection{Lone axis lamination train track (ltt) automata $\mA(G)$ definition}{\label{ss:PffTtautomata}}

%%%%%%%%%%%%%%%%%%%%%%%%%%%%%%%%%%%%%%%%%%%%%%%%%%%%%%%%%%%%%%%%%%%%

There will be an automaton $\mA(G)$ for each ``lone axis ideal Whitehead graph'' $G$: A \emph{rank-$r$ lone axis ideal Whitehead graph} is a finite simplicial graph
\begin{enumerate}
  \item with $2r-1$ vertices, and
  \item such that no connected component has a cut vertex, and
  \item such that each component has $\geq 3$ vertices.
\end{enumerate}

Given a rank-$r$ lone axis ideal Whitehead graph $G$, the \emph{lamination train track (ltt) automaton for $\mG$}, denoted $\mA(G)$, is the disjoint union of the strongly connected components of the finite directed graph $(\mV,\mE)$ defined by
\begin{itemize}
  \item[{\textit{{(Vertices)}}}] $\mV$ is the set of all birecurrent lone axis ltt structures $\mG$ such that $IW(\mG)=G$ and the black edges are labeled with $\{e_1,\cdots, e_n\}$, where $|E\G|=n$ for the underlying graph $\G$ of $\mG$, and
  \item[{\textit{{(Edges)}}}] $\mE\subseteq\mV\times\mV$ is the set of all ordered pairs $(\mG,f\cdot\mG)$ such that $f$ is either a tt-friendly symmetry map or a tt-friendly proper full fold.
\end{itemize} 
\vskip10pt

\begin{rk}
  While not included as part of the definition, it makes more sense in practice to exclude any strongly connected component for which some nontrivial subgraph of the underlying graph is left invariant by all loops in the component. The only real argument for leaving such components in is that they are excluded from consideration anyway by the conditions in Theorem \ref{t:LAautomata_loops_represent1} and Theorem \ref{t:LAautomata_loops_represent2} and are only detected after their construction (by the support of the maps defining the edges).
\end{rk}

\vskip25pt

%%%%%%%%%%%%%%%%%%%%%%%%%%%%%%%%%%%%%%%%%%%%%%%%%%%%%%%%%%%%%%%%%%%%

\subsection{Lone axis tt automata encode all tt representatives}{\label{ss:PffTtautomata}}

%%%%%%%%%%%%%%%%%%%%%%%%%%%%%%%%%%%%%%%%%%%%%%%%%%%%%%%%%%%%%%%%%%%%

\begin{thm}\label{t:LAautomata_loops_represent1} 
Suppose $\mG$ is a lone axis ideal Whitehead graph and $\mA(\mG)$ its ltt automaton. Suppose further that $\mL \colon \mG_0 \xrightarrow{g_{1}}\mG_1 \xrightarrow{g_{2}} \dots \xrightarrow{g_{n}}\mG_n=\mG_0$ is a loop in $\mA(\mG)$. Suppose that $\G_k$ is the underlying graph of $\mG_k$ for each $k$, with $\G_0=\G_n$ also denoted by $\G$. And denote also  $\mG_0=\mG_n$ by $\mG$.

Then $g=g_n\circ\cdots\circ g_0\from\G\to\G$ is a tt map. If $g$ 
\begin{enumerate}
\item takes each turn of $\mG_n=\mG_0$, and
  \item has no PNPs, and
  \item has that its local Whitehead graphs correspond to the colored graph of the ltt structure, and
  \item has a PF transition matrix, and
  \item has that $Image(Dg)$ is precisely the set of purple directions of $g$
\end{enumerate}
then, appropriately marked, $g$ is a fully singular tt representative of an ageometric fully irreducible $\vphi\in\out$ such that $IW(\vphi)=\sqcup SW(\mG,v)$ and the red vertex directions are precisely those not in $Image(Dg^R)$ for a rotationless power $R$. In particular, $\vphi$ is a lone axis fully irreducible outer automorphism.
\end{thm}

\begin{proof} Lemma \ref{lemma:LoopsAreTtMapsLA} covers everything except that  $\vphi$ is a lone axis fully irreducible outer automorphism. Since $\mG(g)$ is a lone axis ltt structure, (ltt-vii) has precisely one red direction, thus precisely one nonperiodic direction, i.e. $DS(g)=1$. By Proposition \ref{p:id}, we then have $i(\vphi)+r-1=\frac{1}{2}$. So $i(\vphi)=\frac{3}{2}-r$. By (ltt-ix), no component of $IW(\mG)$ has a cut vertex. Thus,  $\vphi$ is a lone axis fully irreducible outer automorphism, as desired.\\
\qedhere
\end{proof}

\vskip5pt

\begin{lem}\label{lemma:PartialFoldCojugate} Suppose $g$ is a tt representative of a lone axis fully irreducible $\vphi\in\out$. Then the Stallings fold decomposition of $g$ is partial-fold conjugate to the pff decomposition of a fully singular tt representative of $\vphi$.
\end{lem}

\begin{proof} By following precisely the first three paragraphs of the proof of \cite[Lemma 3.2]{wiggd1}, one sees that $g$ is partial-fold conjugate to a fully singular tt representative $h$ of $\vphi$. Since $h$ is fully singular and lone axis (so only has one illegal turn, thus cannot have a tripod fold), it has a pff decomposition. Since $\vphi$ is a lone axis fully irreducible, this is the only Stallings fold decomposition of $h$. Since $g$ and $h$ represent the same lone axis fully irreducible, this partial-fold conjugation is along their common axis, meaning that in fact the original Stallings fold decomposition of $g$ is partial-fold conjugate to the pff decomposition of $h$.\\
\qedhere
\end{proof}

\vskip5pt

\begin{thm}{\label{t:LAautomata_loops_represent2}} Suppose $g$ is a tt representative of a lone axis fully irreducible $\vphi\in\out$. Then the Stallings fold decomposition of $g$ is partial-fold conjugate to one determining a directed loop in a lone axis ltt automaton, more precisely $\mA(IW(\vphi))$.
\end{thm}

\begin{proof} By Lemma \ref{lemma:PartialFoldCojugate} the Stallings fold decomposition of $g$ is partial-fold conjugate to a pff decomposition 
$$\G_{0} \xrightarrow{h_1} \G_{1} \xrightarrow{h_2} \cdots \xrightarrow{h_{n-1}} \G_{n-1} \xrightarrow{h_n} \G_n$$ 
of some tt representative $h$ of $\vphi$. Let $\mG=\mG(h)$, let $f_k=h_{k,1}\circ h_{n,k+1}$, and let $\mG_k=\mG(f_k)$. Then by Lemma \ref{l:LAlttStructureMaps}, for each $k$, we have $\mG_k=h_{k,1}\cdot \mG$. Further, each $f_k$ is another fully singular tt representative of the same lone axis fully irreducible $\vphi$, so has a lone axis ltt structure with the same ideal Whitehead graph. Thus,
$$\mG=\mG_{0} \xrightarrow{h_1} \mG_{1} \xrightarrow{h_2} \cdots \xrightarrow{h_{n-1}} \mG_{n-1} \xrightarrow{h_n} \mG_n=\mG$$ 
forms a loop in the lone axis ltt automaton $\mA(IW(\vphi))$.\\
\qedhere

\end{proof}

\vskip25pt

%%%%%%%%%%%%%%%%%%%%%%%%%%%%%%%%%%%%%%%%%%%%%%%%%%%%%%%%%%%%%%%%%%%%

%%%%%%%%%%%%%%%%%%%%%%%%%%%%%%%%%%%%%%%%%%%%%%%%%%%%%%%%%%%%%%%%%%%%

\section{Pff train track automata}{\label{s:PffTtautomata}}

%%%%%%%%%%%%%%%%%%%%%%%%%%%%%%%%%%%%%%%%%%%%%%%%%%%%%%%%%%%%%%%%%%%%

%%%%%%%%%%%%%%%%%%%%%%%%%%%%%%%%%%%%%%%%%%%%%%%%%%%%%%%%%%%%%%%%%%%%

%%%%%%%%%%%%%%%%%%%%%%%%%%%%%%%%%%%%%%%%%%%%%%%%%%%%%%%%%%%%%%%%%%%%

\subsection{There are many kinds of red edges}{\label{s:PathologicalExample}} 

%%%%%%%%%%%%%%%%%%%%%%%%%%%%%%%%%%%%%%%%%%%%%%%%%%%%%%%%%%%%%%%%%%%%

The following example highlights three phenomena one may initially believe cannot occur for the ltt structures of a pff decomposition of an ageometric fully irreducible outer automorphism, but are proved to occur in Proposition \ref{prop:pathologies}.

\begin{ex}\label{e:PathologicalExample}
Let $\mathfrak{g}=g_{28}\circ\dots\circ g_1$, where $\sigma$ reverses the orientation on $b$ (i.e. $b\mapsto\overline{b}$ and $\overline{b}\mapsto b$) and then we define $g_{14+k}:=\sigma^{-1}\circ g_k \circ \sigma$ for each $k\in\{1,2,\dots, 14\}$.
%\begin{figure}[H]
\begin{center}
\noindent\includegraphics[width=6.7in]{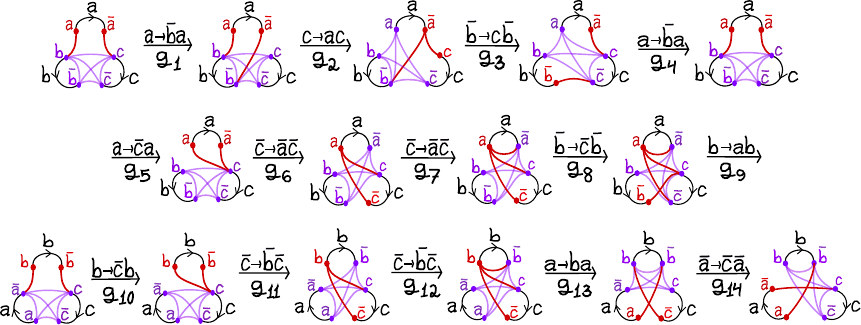}
\end{center}

\end{ex}

\vskip10pt

\begin{prop}[Pff Stallings fold decomposition pathologies]\label{prop:pathologies}
Suppose $\vphi\in\out$ is ageometric fully irreducible with a pff decomposition $\G_{0} \xrightarrow{g_1} \G_{1} \xrightarrow{g_2} \cdots \xrightarrow{g_{n-1}} \G_{n-1} \xrightarrow{g_n} \G_n$ of a fully singular tt representative $g\from\G\to\G$. Then each of the following may occur:\\
1. the number of red edges in the ltt structure may vary during the pff decomposition, and\\
2. a red edge may connect 2 red directions, and\\
3. a red direction may be contained in 2 red edges.\
\end{prop}

\smallskip

\begin{proof} It suffice to show that the map $\mathfrak{g}$ above is a tt representative of an ageometric fully irreducible outer automorphism with the indicated ltt structure. We use Proposition \ref{prop:FIC} for the former.

All periodic directions of $\mathfrak{g}$ are fixed and the only directions not in the image of $D\mathfrak{g}$ are $a$ and $\overline{a}$. The illegal turns for $\mathfrak{g}$ are $\{a,\overline{b}\}$ and $\{\overline{a},\overline{c}\}$, as one can see because $Dg_1$ identified the directions $a$ and $\overline{b}$ and then the only other fold $g_k$ identifying 2 directions in the image of $D(g_{k+1}\circ g_1)$ is $g_6$.

Using \cite[Lemma 1]{automaton}, the first two charts below give the turns taken to be those of the ltt structure. Note that the only local Whitehead graph is indeed connected. By tracing edge images, one can also see that the transition matrix is PF. We are thus left to show that $\mathfrak{g}$ has no PNPs.

\begin{table}[h]
\begin{tabular}{|l|l|l|l|l|l|l|l|l|l|}
\hline 
 &  &  &  &  &  &  &  &  & \\
 & \quad $g_1$ & \quad $g_2$ & \quad $g_3$ & \quad $g_4$  & \quad $g_5$ & \quad $g_6$ & \quad $g_7$ & \quad  $g_8$  &\quad  $g_9$ \\
 
 & $a\mapsto\ol{b}a$ & $c\mapsto ac$ & $\ol{b}\mapsto c\ol{b}$ & $a\mapsto\ol{b}a$ &  $a\mapsto\ol{c}a$ & $\ol{c}\mapsto \ol{a}\ol{c}$ & $\ol{c}\mapsto \ol{a}\ol{c}$  & $\ol{b}\mapsto \ol{c}\ol{b}$  & $b\mapsto ab$ \\
 
 &  &  &  &    &  &  &  &  &   \\
 \hline 
 &  &  &  &   &  &  &  &   &   \\
\text{New} & $\{a,b\}$ & $\{\ol{a},c\}$ & $\{\ol{b},\ol{c}\}$ & $\{a,b\}$ & $\{a,c\}$ & $\{a,\ol{c}\}$ & $\{a,\ol{c}\}$ & $\{c,\ol{b}\}$  & $\{\ol{a},b\}$  \\
 &  &  &  &   &  &  &  &   &  \\
\hline 
 &  &  &  &   &  &  &  &   &  \\
$Dg_k(\tau_k)$ &  & $\{a,b\}$  &  $\{a,b\}$ & $\{\ol{b},b\}$ & $\{\ol{b},b\}$ & $\{\ol{b},b\}$, $\{a,c\}$  & $\{\ol{b},b\}$, $\{\ol{a},c\}$  & $\{\ol{c},b\}$, $\{\ol{a},c\}$  &  $\{\ol{a},a\}$, $\{\ol{a},c\}$ \\
 &  &  &  &   &  &  &  &   &  \\
 &  &  & $\{\ol{a},c\}$ & $\{\ol{b},\ol{c}\}$ & $\{\ol{b},\ol{c}\}$ & $\{\ol{b},\ol{a}\}$ & $\{\ol{b},\ol{a}\}$, $\{\ol{a},b\}$ & $\{\ol{c},\ol{a}\}$, $\{\ol{a},b\}$   & $\{\ol{c},\ol{a}\}$ \\
 &  &  &  &   &  &  &  &   &  \\
 &  &  &  & $\{\ol{a},c\}$ & $\{\ol{a},c\}$ &  $\{\ol{a},c\}$  &  $\{a,c\}$, $\{\ol{a},a\}$ &  $\{a,c\}$, $\{\ol{a},a\}$   &  $\{a,c\}$  \\
 &  &  &  &   &  &  &  &   &  \\
 &  &  &  &   & $\{\ol{c},b\}$ & $\{\ol{a},b\}$  &  &  $\{a,\ol{c}\}$  &  $\{a,\ol{c}\}$  \\
 &  &  &  &   &  &  &  &   &  \\
\hline 
\end{tabular}
\end{table}

\begin{table}[h]
\begin{tabular}{|l|l|l|l|l|l|}
\hline 
 &  &  &  &   &    \\
 & \quad $g_{10}$ & \quad $g_{11}$ & \quad $g_{12}$ & \quad $g_{13}$  & \quad $g_{14}$  \\
 
 & $b\mapsto\ol{c}b$ & $\ol{c}\mapsto \ol{b}\ol{c}$ & $\ol{c}\mapsto \ol{b}\ol{c}$ & $a\mapsto ba$ &  $\ol{a}\mapsto\ol{c}\ol{a}$  \\
 
 &  &  &  &    &    \\
 \hline 
 &  &  &  &   &    \\
\text{New}  & $\{b,c\}$ & $\{b,\ol{c}\}$ & $\{b,\ol{c}\}$ & $\{a,\ol{b}\}$  & $\{\ol{a},c\}$  \\
 &  &  &  &   &    \\
\hline 
 &  &  &  &   &    \\
$Dg_k(\tau_k)$  & $\{\ol{a},a\}$, $\{\ol{a},c\}$ & $\{\ol{a},a\}$, $\{\ol{a},c\}$  & $\{\ol{a},a\}$, $\{\ol{a},c\}$, $\{\ol{b},b\}$  & $\{\ol{a},b\}$, $\{\ol{a},c\}$, $\{\ol{b},b\}$  &  $\{\ol{c},b\}$, $\{\ol{c},c\}$, $\{\ol{b},b\}$ \\
 &  &  &  &   &   \\
 & $\{\ol{c},\ol{a}\}$, $\{a,c\}$ & $\{\ol{b},\ol{a}\}$, $\{a,c\}$  & $\{\ol{b},\ol{a}\}$, $\{a,c\}$ & $\{\ol{b},\ol{a}\}$, $\{b,c\}$ & $\{\ol{b},\ol{c}\}$, $\{b,c\}$  \\
 &  &  &  &   &   \\
 & $\{a,\ol{c}\}$ & $\{a,\ol{b}\}$, $\{b,c\}$   & $\{a,\ol{b}\}$, $\{b,c\}$  &  $\{b,c\}$, $\{b,\ol{c}\}$  & $\{b,c\}$, $\{a,\ol{b}\}$  \\
 &  &  &  &   &   \\
\hline 
\end{tabular}
\end{table}

To show there are no PNPs we use the third chart below. We repeatedly use Lemma \ref{l:pnpelimination} and Lemma \ref{l:longinps} for this.

\begin{table}[h]
\begin{tabular}{|l|l|l|l|l|l|l|}
\hline 
 &  &  &  &   &  &   \\
 & \quad $g_1$ & \quad $g_2$ & \quad $g_3$ & \quad $g_4$  & \quad $g_5$ & \quad $g_6$  \\
 
 & $a\mapsto\ol{b}a$ & $c\mapsto ac$ & $\ol{b}\mapsto c\ol{b}$ & $a\mapsto\ol{b}a$ &  $a\mapsto\ol{c}a$ & $\ol{c}\mapsto \ol{a}\ol{c}$ \\
 
 &  &  &  &    &  &     \\
 \hline 
 &  &  &  &   &  &     \\
 & \quad $f_0=\mathfrak{g}$ & \quad \quad $f_1$ & \quad \quad $f_2$ & \quad \quad $f_3$  & \quad \quad $f_4$  & \quad \quad $f_5$    \\
 
 &  &  &  &   &  &     \\
 \hline 
 &  &  &  &    &  &    \\
\text{illegal turns} & $\{a,\ol{b}\}$, $\{\ol{a},\ol{c}\}$ & $\{a,c\}$, $\{\ol{a},\ol{c}\}$ & $\{\ol{b},c\}$, $\{\ol{a},\ol{c}\}$ & $\{a,\ol{b}\}$, $\{\ol{a},\ol{c}\}$ & $\{a,\ol{c}\}$, $\{\ol{a},\ol{c}\}$ & $\{\ol{a},\ol{c}\}$, $\{a,\ol{b}\}$ \\
 &  &  &  &    &  &   \\
\hline
\hline 
 &  &  &  &   &  &   \\
 & \quad $\mathfrak{g}_{1,1}=g_1$ & \quad \quad $\mathfrak{g}_{2,1}$ & \quad \quad $\mathfrak{g}_{3,1}$ & \quad $\mathfrak{g}_{4,1}$   & \quad $\mathfrak{g}_{5,1}$   &  \quad $\mathfrak{g}_{6,1}$   \\
 
 &  &  &  &    &  &    \\
 \hline 
 &  &  &  &    &  &    \\
 & \quad $a\mapsto \ol{b}a$ & \quad $a\mapsto \ol{b}a$ & \quad $a\mapsto c\ol{b}a$ & \quad $a\mapsto c\ol{b}\ol{b}a$  & \quad $a\mapsto c\ol{b}\ol{b}\ol{c}a$  & \quad $a\mapsto ca\ol{b}\ol{b}\ol{a}\ol{c}a$  \\
 
 & \quad $b\mapsto b$  & \quad $b\mapsto b$  & \quad $b\mapsto b\ol{c}$  & \quad $b\mapsto b\ol{c}$   & \quad $b\mapsto b\ol{c}$  & \quad $b\mapsto b\ol{a}\ol{c}$   \\
 
 & \quad $c\mapsto c$  & \quad $c\mapsto ac$  & \quad $c\mapsto ac$  & \quad $c\mapsto \ol{b}ac$    & \quad $c\mapsto \ol{b}\ol{c}ac$  & \quad $c\mapsto \ol{b}\ol{a}\ol{c}aca$   \\
 &  &  &  &    &  &    \\
\hline 
\end{tabular}
\end{table}

Suppose for the sake of contradiction that $\rho$ were a PNP. By  Lemma \ref{l:pnpelimination}, $\rho=\ol{\rho_1}\rho_2$ for some dangerous long turn $\{\rho_1,\rho_2\}$.

The 2 illegal turns of $\mathfrak{g}$ are $\{a,\overline{b}\}$ and $\{\overline{a},\overline{c}\}$. Suppose first the dangerous long turn illegal turn was $\{\ol{a},\ol{c}\}$. Since \\
\indent $\mathfrak{g}_{6,1}(\ol{a})=\ol{a}cabb\ol{a}\ol{c}$ and\\
\indent $\mathfrak{g}_{6,1}(\ol{c})=\ol{a}\ol{c}\ol{a}cab$,\\
cancellation ends with $\{c,\overline{c}\}$, which cannot be a pff decomposition illegal turn in light of Lemma \ref{lem:graphmap}d. This contradicts Lemma \ref{l:pnpelimination}.

Now suppose the dangerous long turn illegal turn was $\{a,\overline{b}\}$. Then, in light of Lemma \ref{l:pnpelimination}b, there exist edges $e_i,e_j'\in E\G$ so that $\rho_1=a e_2\dots e_n$ and $\rho_2=\overline{b} e_2'\dots e_m'$. Note that each turn of $\rho_1$ and $\rho_2$ must be $\mathfrak{g}$-taken. Now\\
\indent $g_1(\rho_1) = g_1(a e_2\dots e_n) = \overline{b}a g_1(e_2)\dots g_1(e_n)$ and\\
\indent $g_1(\rho_2) = g_1(\overline{b} e_2'\dots e_m') = \ol{b} g_1(e_2')\dots g_1(e_m')$.\\
So we need that $\{a, g_1(e_2')\}$ is either degenerate or an illegal turn for $f_1$, i.e. $g_1(e_2')=a$ or $g_1(e_2')=c$. Since $a\notin Image(Dg_1)$ and only $Dg_1(c)=c$, we have $e_2'=c$, i.e. $\rho_2 = \ol{b}c e_3'\dots e_m'$. So

$\mathfrak{g}_{2,1}(\rho_1) = \mathfrak{g}_{2,1}(a e_2\dots e_n) = \ol{b}a ~\mathfrak{g}_{2,1}(e_2)\dots \mathfrak{g}_{2,1}(e_n)$ and

$\mathfrak{g}_{2,1}(\rho_2) = \mathfrak{g}_{2,1}(\ol{b}c e_3'\dots e_m') = \ol{b}ac ~\mathfrak{g}_{2,1}(e_3')\dots \mathfrak{g}_{2,1}(e_m')$.\\
So we need that $\{\mathfrak{g}_{2,1}(e_2), c \}$ is either degenerate or an illegal turn for $f_2$, i.e. $\mathfrak{g}_{2,1}(e_2)=\ol{b}$ or $\mathfrak{g}_{2,1}(e_2)=c$. 
Since $c\notin Image(D\mathfrak{g}_{2,1})$, only $D\mathfrak{g}_{2,1}(a), D\mathfrak{g}_{2,1}(\ol{b}) =\ol{b}$, 
and $\{\ol{a},a\}\notin\tau_{\infty}(\mathfrak{g})$, we have  $e_2=\ol{b}$, i.e. $\rho_1=a\ol{b} e_3\dots e_n$. So

$\mathfrak{g}_{3,1}(\rho_1) = \mathfrak{g}_{3,1}(a\ol{b} e_3\dots e_n) = c\ol{b}ac\ol{b} ~\mathfrak{g}_{3,1}(e_3)\dots \mathfrak{g}_{3,1}(e_n)$ and

$\mathfrak{g}_{3,1}(\rho_2) = \mathfrak{g}_{3,1}(\ol{b}c e_3'\dots e_m') = c\ol{b}ac ~\mathfrak{g}_{3,1}(e_3')\dots \mathfrak{g}_{3,1}(e_m')$.\\
So we need $\{\ol{b}, \mathfrak{g}_{3,1}(e_3') \}$ is either degenerate or an illegal turn for $f_3$, i.e. $\mathfrak{g}_{3,1}(e_3')=\ol{b}$ or $\mathfrak{g}_{3,1}(e_3')=a$. 
Since $\ol{b}\notin Image(D\mathfrak{g}_{3,1})$ and only $D\mathfrak{g}_{3,1}(c) =a$, we have $e_3'=c$, i.e. $\rho_2=\ol{b}cc e_4'\dots e_m'$. So

$\mathfrak{g}_{4,1}(\rho_1) = \mathfrak{g}_{4,1}(a\ol{b} e_3\dots e_n) = c\ol{b}\ol{b}ac\ol{b} ~\mathfrak{g}_{4,1}(e_3)\dots \mathfrak{g}_{4,1}(e_n)$ and

$\mathfrak{g}_{4,1}(\rho_2) = \mathfrak{g}_{4,1}(\ol{b}cc e_4'\dots e_m') = c\ol{b}\ol{b}ac\ol{b}ac ~\mathfrak{g}_{4,1}(e_4')\dots \mathfrak{g}_{4,1}(e_m')$.\\
Cancellation ends with $\{\mathfrak{g}_{4,1}(e_3), a \}$. So we need $\{\mathfrak{g}_{4,1}(e_3), a \}$ is either degenerate or an illegal turn for $f_4$, i.e. $\mathfrak{g}_{4,1}(e_3)=a$ or $\mathfrak{g}_{4,1}(e_3)=\ol{c}$. 
Since $a\notin Image(D\mathfrak{g}_{5,1})$ and only $D\mathfrak{g}_{4,1}(\ol{c}) =\ol{c}$, we have $e_3=\ol{c}$, i.e. $\rho_1=a\ol{b}\ol{c} e_4\dots e_n$. So

$\mathfrak{g}_{5,1}(\rho_1) = \mathfrak{g}_{5,1}(a\ol{b}\ol{c} e_4\dots e_n) = c\ol{b}\ol{b}\ol{c}ac\ol{b}\ol{c}\ol{a}cb ~\mathfrak{g}_{5,1}(e_4)\dots \mathfrak{g}_{5,1}(e_n)$ and

$\mathfrak{g}_{5,1}(\rho_2) = \mathfrak{g}_{5,1}(\ol{b}cc e_4'\dots e_m') = c\ol{b}\ol{b}\ol{c}ac\ol{b}\ol{c}ac ~\mathfrak{g}_{5,1}(e_4')\dots \mathfrak{g}_{5,1}(e_m')$.\\
Cancellation ends with $\{a,\ol{a}\}$, which again cannot be a pff decomposition illegal turn in light of Lemma \ref{lem:graphmap}d. So we have reached a contradiction with Lemma \ref{l:pnpelimination}.  

So $\rho$ could not have existed and $\mathfrak{g}$ has no PNPs. 
\qedhere
\end{proof}

\vskip15pt

%%%%%%%%%%%%%%%%%%%%%%%%%%%%%%%%%%%%%%%%%%%%%%%%%%%%%%%%%%%%%%%%%%%%

\subsection{Fully singular pff ltt automata definition}{\label{ss:PffTtautomatadfn}}

%%%%%%%%%%%%%%%%%%%%%%%%%%%%%%%%%%%%%%%%%%%%%%%%%%%%%%%%%%%%%%%%%%%%

There will be an automaton $\mA(G)$ for each ``fully singular ideal Whitehead graph'' $G$: A \emph{rank-$r$ fully singular ideal Whitehead graph} is a finite simplicial graph with 
\begin{enumerate}
\item $1\leq c\leq 2r-1$ connected components each having $\geq 3$ vertices and
  \item $2r-1 \leq V\G \leq 6r-5$ vertices total.
\end{enumerate}

Given a rank-$r$ fully singular ideal Whitehead graph $G$, the \emph{lamination train track (ltt) automaton for $G$}, denoted $\mA(G)$, is the disjoint union of the strongly connected components of the finite directed graph $(\mV,\mE)$ defined by
\begin{itemize}
  \item[{\textit{{(Vertices)}}}] $\mV$ is the set of all birecurrent abstract ltt structures $\mG$ such that $IW(\mG)=G$ and the black edges are labeled with $\{e_1,\cdots, e_n\}$, where $|E\G|=n$ for the underlying graph $\G$ of $\mG$, and
  \item[{\textit{{(Edges)}}}] $\mE\subseteq\mV\times\mV$ is the set of all ordered pairs $(\mG,f\cdot\mG)$ such that $f$ is either a tt-friendly symmetry map or a tt-friendly proper full fold.
\end{itemize} 
\vskip10pt

%%%%%%%%%%%%%%%%%%%%%%%%%%%%%%%%%%%%%%%%%%%%%%%%%%%%%%%%%%%%%%%%%%%%

\subsection{Fully singular pff ltt automata encode all fully singular pff decompositions}{\label{ss:PffTtautomata}}

%%%%%%%%%%%%%%%%%%%%%%%%%%%%%%%%%%%%%%%%%%%%%%%%%%%%%%%%%%%%%%%%%%%%

\begin{thm}\label{t:Pffautomata_loops_represent1} 
Suppose $\mG_0 \xrightarrow{g_{1}}\mG_1 \xrightarrow{g_{2}} \dots \xrightarrow{g_{n}}\mG_n=\mG_0$ is a loop in $\mA(G)$ for some fully singular ideal Whitehead graph $G$.
Then 
\begin{itemize}
  \item[(a.)] $g_n\circ\cdots\circ g_1\from\G\to\G$ is a tt map and
  \item[(b.)] each turn of $\tau_{\infty}(g)$ is represented by a colored edge in $\mG_0$.
\end{itemize} 
If $g$ additionally
\begin{enumerate}
%\item takes each turn of $\mG$, and
  \item has no PNPs, and
  \item  has that each $LW(g,v)$ is graph isomorphic to the appropriate components of the colored graph of the ltt structure, and
  \item has a PF transition matrix, and
  \item has $\geq 3$ periodic directions at each vertex 
\end{enumerate}
then, appropriately marked, $g$ is a fully singular tt representative of an ageometric fully irreducible $\vphi\in\out$.

If further yet, 
\begin{enumerate}\addtocounter{enumi}{4}
\item $Image(Dg)$ is precisely the set of purple directions of $\mG_0$,
\end{enumerate}
then $g$ is a fully singular tt representative of an ageometric fully irreducible $\vphi\in\out$ such that $IW(\vphi)=G$. 
\end{thm}

\begin{proof} This theorem is basically a direct application of Lemma \ref{lemma:LoopsAreTtMapsLA}. 
\qedhere

\end{proof}

\vskip10pt

\begin{thm}{\label{t:Pffautomata_loops_represent2}} Suppose $g$ is a fully singular tt representative of an ageometric irreducible $\vphi\in\out$. And suppose that 
$\G_0 \xrightarrow{g_{1}}\G_1 \xrightarrow{g_{2}} \dots \xrightarrow{g_{n}}\G_n=\G_0$ is a pff decomposition of $g$. Let $\mG_k:=\mG(f_k)$ for each $k$, and $\mG:=ltt(g)$. Then
$\mG_0 \xrightarrow{g_{1}}\mG_1 \xrightarrow{g_{2}} \dots \xrightarrow{g_{n}}\mG_n=\mG_0$ is a directed loop in a pff ltt automaton, more precisely $\mA(IW(\vphi))$.
\end{thm}

\begin{proof} 
By Lemma \ref{lemma:MapsOfPffLttStructures}, $\mG_k=g_{k,1}\cdot\mG=\mG(f_k)$.
Thus, by Lemma \ref{lemma:FSlttstructures} and recalling that the rotationless index is an outer automorphism conjugacy class invariant, each $\mG_k$ is a birecurrent abstract ltt structure with underlying graph $\G_k$. In fact, since the ideal Whitehead graph is also a conjugacy class invariant, $IW(\mG_k)=IW(\mG)$ for each $k$.

That the ideal Whitehad graph satisfies (1)-(2) of the definition of a rank-$r$ fully singular ideal Whitehead graph follows from a computation using that each vertex has $\geq 3$ directions, one vertex has $\geq 4$ directions, the Euler characteristic satisfies $1-r=\chi(\G)=|V\G|-|E\G|$, and $|\mD\G|=2|E\G|$.

That the proper full folds and homeomorphism of a pff decomposition are, respectively, tt-friendly proper full folds and a tt-friendly symmetry map follows from Lemma \ref{l:PffdecompsTTfriendly}.\\
\qedhere

\end{proof}

\vskip15pt

%%%%%%%%%%%%%%%%%%%%%%%%%%%%%%%%%%%%%%%%%%%%%%%%%%%%%%%%%%%%%%%%%%%%

\subsection{Paths in fully singular pff ltt automata yield geodesics in $\os$}{\label{ss:PathsAreGeodesic}}

%%%%%%%%%%%%%%%%%%%%%%%%%%%%%%%%%%%%%%%%%%%%%%%%%%%%%%%%%%%%%%%%%%%%

\smallskip

\begin{thm}\label{t:PathsAreGeodesic} 
Suppose $\dots\mG_{-1}\xrightarrow{g_{-1}}\mG_0 \xrightarrow{g_{1}}\mG_1 \xrightarrow{g_{2}} \dots$ is a bi-infinite path in $\mA(G)$ for some fully singular rank-$r$ ideal Whitehead graph $G$ such that each $g_k$ is a proper full fold. Then the sequence of folds $\dots g_{-2}, g_{-1}, g_0, g_1, g_2\dots$ defines a geodesic in $\os$.
\end{thm}

\begin{proof} 
We will want to apply Proposition \ref{foldLineGeodesics}, but Proposition \ref{foldLineGeodesics} only applies to fold rays.

The complication to defining a metric on the graphs only arises in the folding direction because one might worry that one is trying to fold a shorter edge over a longer edge. This aspect of determining a path in $\os$ is resolved by first assigning edge-lengths for the fold sequence $g_0, g_1, g_2\dots$ and then adding (and renormalizing) lengths in the reverse direction.

Now suppose that there is some integer $k$ so that $g_k, g_{k+1}, g_{k+2}\dots$ does not define a geodesic. We reach a contradiction by finding a conjugacy class $\alpha$ in $F_r$ so that, for each $i\geq k$, the realization $\alpha_i$ of $\alpha$ in $\G_{i-1}$ is not folded by $g_i$. We take a smooth loop $\ell$ in $\mG_{k-1}$ that contains every colored edge in $\mG_{k-1}$. This loop $\ell$ determines a loop in $\G_{k-1}$ and $g_k(\ell)$ determines a loop in $\mG_{i-1}$, so cannot be folded by $g_i$ because colored turns are never folded in $\mA(G)$. The loop $\ell$ defines $\alpha$.
\\
\qedhere

\end{proof}

\vskip35pt

%%%%%%%%%%%%%%%%%%%%%%%%%%%%%%%%%%%%%%%%%%%%%%%%%%%%%%%%%%%%%%%%%%%%

%%%%%%%%%%%%%%%%%%%%%%%%%%%%%%%%%%%%%%%%%%%%%%%%%%%%%%%%%%%%%%%%%%%%

\bibliographystyle{alpha}

\bibliography{PaperRefs}

\vskip25pt

\end{document}